\documentclass[12pt]{amsart}

\usepackage{stile_article_resultant_final-arxiv}

\title[Lower bounds for the multiplicity of the resultant]{Multigraded Koszul Complexes,\\ Filter-Regular Sequences and \\Lower Bounds for the Multiplicity of the Resultant}
\author{Luca Ghidelli}
\date{\today}

\address{150 Louis-Pasteur Private, Office 608, Department of Mathematics and Statistics, University of Ottawa, Ottawa ON K1N 9A7, Canada}
\email{{luca.ghidelli@uottawa.ca}}

\DeclareRobustCommand{\SkipTocEntry}[4]{}

\begin{document}


\begin{abstract}
The R\'emond resultant attached to a multiprojective variety and a sequence of multihomogeneous polynomials is a polynomial form in the coefficients of the polynomials, which vanishes if and only if the polynomials have a common zero on the variety. 
We demonstrate that this resultant can be computed as a Cayley determinant of a multigraded Koszul complex, proving a key stabilization property with the aid of local Hilbert functions and the notion of filter-regular sequences. 
Then we prove that the R\'emond resultant vanishes, under suitable hypotheses, with order at least equal to the number of common zeros of the polynomials. 
More generally, we estimate the multiplicity of resultants of multihomogeneous polynomials along prime ideals of the coefficient ring, thus considering for example the order of $p$-adic vanishing. 
Finally, we exhibit a corollary of this multiplicity estimate in the context of interpolation on commutative algebraic groups, with applications to Transcendental Number Theory.
\end{abstract}

\maketitle

\section*{Introduction} \label{sec:intro}


\addtocontents{toc}{\SkipTocEntry}
The theory of resultants is an old branch of Mathematics which provides important tools, both computational and theoretical, in many other fields.
One of the most classical versions of a resultant, named after Macaulay \cite{Macaulay}, is defined for 
a sequence $\Seq=(\seq_0,\ld,\seq_\r)$ of $\r+1$ homogeneous polynomials 
in $\r+1$ variables $\x=(x_0,\ld,x_\r)$ over a field $\kk$.
The Macaulay resultant of $\Seq$ is an irreducible polynomial of the unknown coefficients of $\Seq$ uniquely determined, up to a multiplication by a constant, by the following property: it vanishes if and only if the polynomials in $\Seq$ admit a nontrivial common zero over an algebraic closure of $\kk$. 
It turns out that such implicit characterization gives rise to a mathematical object that can be computed explicitly \cite{physics,MatricesElimination,Cox:algorithms} and that satisfies several remarkable properties \cite{Cox:using,SturmfelsIntroduction,Jouanolou:formalisme,Jouanolou:invariant}. 
In this paper we discuss the following statement, together with its generalizations and applications. 

\begin{theorem}\label{thm:main:macaulay}
	Suppose that the polynomials $\Seq$ have exactly $N$ common roots, counting multiplicities. Then the Macaulay resultant, considered as a polynomial function, vanishes with multiplicity at least $N$ when specialized at the coefficients of $\Seq$. 
\end{theorem}

This is a useful property, 
that can be interpreted either as a multiplicity estimate for the resultant, or as an upper bound for the number of solutions to the nonlinear system given by $\seq_i = 0$ for $i=0,\ld,\r$. 
This theorem can be shown with a variety of methods when $\r=1$: for instance one may use a formula of Poisson \cite{Poisson} that expressess the resultant of $\Seq=(\seq_0,\seq_1)$, up to a nonzero multiplicative constant, as the symmetric polynomial
$$
 \prod_{i=1}^{d_0}\prod_{j=1}^{d_1} (\alpha_{0,i}-\alpha_{1,j})
$$
in the roots of $f_0$ and $f_1$. 
A version of \cref{thm:main:macaulay} was proved by Roy \cite[Theorem 5.2]{Roy2013} for all $\r\geq 1$, under the hypothesis that $k=\C$ and all the polynomials have equal degree. 
If the polynomials $\Seq$ have integer coefficients, there is an interesting arithmetic analogue of \cref{thm:main:macaulay} given by Chardin \cite{Chardin1993}. 
In this setting one often normalizes the irreducible polynomial that defines the Macaulay resultant to have integer coefficients, so the resultant of $\Seq$ is an integer $R$. 
Chardin proves, under suitable hypotheses, that if $p$ is a prime number and the polynomials $\Seq$ have $N$ common zeros modulo $p$, then $p^N\divides R$. 
In fact we observe in \cref{rmk:chardin:enhancement} that with this point of view it is possible to prove a statement that is stronger than Chardin's. 
In \cite{SchejaStorch} Scheja and Storch treat  \cref{thm:main:macaulay} and  its arithmetic analogue as expressions of the same phenomenon, by working on polynomial algebras over integrally closed Noetherian domains and using a sufficiently general notion of ``vanishing''. 
For a somewhat unsimilar study of the multiplicity of the different, which can be thought as a geometric analog of the resultant of a gradient system of equations, we refer to Aluffi and Cukierman \cite{AluffiCukierman}. 
The goal of this paper is to prove a multiplicity estimate for multigraded Chow forms, known also as R\'emond resultants. 
The main purpose is to deduce a corollary with potential applications in Algebraic Independence and Transcendental Number Theory.

The Macaulay resultant is only one of several notions in the rich theory of resultants. 
Many of the approaches to this theory are algebraic and express the resultant by means of determinantal formulas, we refer to \cite{Demazure,Jouanolou:formalisme} without the intent of completeness. 
However most generalizations come from geometric interpretations of the concept of resultants. 
In Algebraic Geometry one has the notion of Chow forms attached to arbitrary projective subvarieties \cite{PhilipponCh6}. 
For comparison, the Macaulay resultant is a Chow form for the projective variety $\PP^\r_\kk$. 
Chow forms are important Intersection-Theoretic invariants and are applied in the theory of Heights \cite{Philippon}. 
Moreover, the Chow forms of toric varieties \cite{toric:Cox,toric:Fulton} are often called sparse resultants  and are important for computational reasons \cite{Sturmfels1994Newton,MatricesElimination}.  
Further generalizations with a geometric flavour, such as mixed resultants, can be found in the monograph of Gelfand-Kapranov-Zelevinski \cite{GKZ94}. 
Multigraded Chow forms, or R\'emond resultants, are attached to sequences of multihomogeneous polynomials $\Seq$ and multiprojective varieties/schemes $V\subseteq \PP^{n_1}_\kk\times\dots\times\PP^{n_\q}_\kk$. 
The R\'emond resultant of $(\Seq,V)$ is a (not necessarily irreducible \cite[Example 1.31]{Nullstellen}) polynomial of the unknown coefficients of $\Seq$ with the following property: it vanishes if and only if the polynomials in $\Seq$ admit a nontrivial common zero in $V$ over an algebraic closure of $\kk$. 
This is a notion of resultants which encapsulates most of the above definitions \cite[Remark 1.39]{Nullstellen}. 

In order to prove the aforementioned multiplicity estimate, we show that the resultants of R\'emond can be computed as Cayley determinants of suitable multigraded Koszul complexes. 
This addresses a gap in the literature and it has other consequences. 
For example, it implies that the multiprojective resultants satisfy several classical formulas, such as the one that expresses the resultant as a gcd of the maximal minors of the Sylvester map \cite[Theorem 34, Appendix A]{GKZ94}.

A notion of resultant for multihomogeneous polynomials is important in applications, especially when the set of variables $\x=(x_0,\ld,x_\r)$ decomposes naturally in independent subcollections $\x=\x^{(1)}\cup\ld\cup\x^{(q)}$.   
This is often the case, for example, in Transcendental and Algebraic Independence Theory, for which we refer to the book \cite{AlgebraicIndependence}. 
In this theory one typically starts with a tuple of numbers $\underline \xi$ that one wishes to prove algebraically independent (or $\Q$-linearly independent, or else), and then takes  suitable combinations of $\underline\xi$ to fabricate a set of points $\Sigma\subseteq G(\C)$ of an algebraic group $G=G_1\times\dots\times G_q$. 
One then assumes that the numbers $\underline\xi$ are not algebraically independent and constructs so-called auxiliary functions $\seq_i$ that vanish with high order at all points of $\Sigma$, in hope to find a contradiction.  
For a detailed account on this method, we refer to the book of Waldschmidt \cite{Waldschmidt2000}. 
If $\r$ is the dimension of $G$ and the auxiliary functions $\Seq=(\seq_0,\ld,\seq_\r)$ are polynomials, one may consider their resultant. 
Since the polynomials $\seq_i$ vanish simultaneously on $\Sigma$ with high multiplicity, it follows from a suitable version of \cref{thm:main:macaulay} that their resultant vanishes with high multiplicity as well: we explore this matter in more detail in \cref{sec:application}. 
This construction might be seen as a way to package the information of several auxiliary polynomials $\Seq$ into a single ``larger'' auxiliary polynomial, namely their resultant. 
For examples of how the information on the multiplicity of the resultants is used to derive Diophantine results, in the context of interpolation on the commutative algebraic group $\mathbb G_a\times\mathbb G_m$, we refer to \cite{Roy2013,Ghidelli2015,Nguyen2016}.

\section*{Plan of the paper and methodology} 
\label{methodology}

The paper is subdivided as follows. 
In \cref{sec:multi} we review the basic definitions and results we need from multigraded Commutative  Algebra and multiprojective Geometry. 
In \cref{sec:res} we introduce the multigraded Koszul complex, we define the resultant as the determinant of sufficiently high multidegree slices of this complex, and we show that this definition coincides with that of R\'emond. 
In \cref{sec:main} we prove our main multiplicity estimates and finally in \cref{sec:application} we present an application to the interpolation theory on commutative algebraic groups. 

As we already remarked, the resultant is an algebraic invariant with a geometric interpretation. 
The most classical versions of the theory of resultants are formulated for polynomial algebras $A[\x]$, whereas some abstract geometric resultants are attached to $\mcl O_X$-vector bundles, and their twists, over $\r$-dimensional schemes \cite{GKZ94}. 
One may find a middle ground by considering the ``M-resultant'' attached to an $A[\x]$-module. 
%
%
This unorthodox approach is the one that we adopt in this paper, see \cref{rmk:hypotheses} for a discussion on the hypotheses on $A[\x]$ and $M$. 
One reason for this choice is the fact that the module ``$M$'' that lurks under the definition of a (multigraded) Chow form does not necessarily have the structure of a polynomial algebra, if the underlying scheme is not a toric variety. 
Nevertheless, the multihomogeneous components of this multigraded module are free (cfr. \cref{sec:multi:poly}), and this hypothesis turns out to be sufficient to guarantee the validity of our constructions. 
Therefore it is natural and not more difficult to allow $M$ to be an essentially arbitrary module with this property, instead of restricting it only to the modules that arise in the construction of Chow forms. 
See \cref{rmk:recovery} for the recovery of the classical theory and the theory of R\'emond,  and see  \cref{rmk:schemes} for a comparison with the geometric generalizations. 

The algebraic theory of resultants is intimately related with the notion of regular sequences, and this reflects their fundamental intersection-theoretic nature. 
In this paper we use instead the more general notion of filter-regular elements: these are like regular elements that ``disregard'' the irrelevant associated primes of the module (\cref{filter-regular:relevant}). 
Filter-regular elements and sequences are thus natural objects in multigraded Commutative Algebra and in fact it turns out that the theory of regular sequences alone is ill-suited for developing the theory of multigraded resultants/Chow forms. 
The geometric reason is that the ``multiaffine cone'' of a multiprojective variety may have singularities at the ``multi-vertex'' that are not Cohen-Macaulay: this may prevent the very existence of regular sequences with right length, see \cref{filter:regular:why}. 

Let us briefly discuss the definition of the M-resultant attached to an $A[\x]$-module $M$ and a filter-regular sequence $\Seq$. 
One approach, adopted e.g. by R\'emond \cite{RemondCh5},  is to define it as the annihilant form (or content, see \cref{def:ann}) of any multihomogeneous component of module $M/(\Seq)M$ with sufficiently high multidegree.   
Notice that the $M/(\Seq)M$ is the cokernel of the multigraded linear map
\begin{equation}\label{def:sylvester}
\begin{aligned}
\de_1 : M\times \dots \times M 	&\to M \\
(m_0, \dots, m_\r) 			&\mapsto \seq_0m_0 + \dots + \seq_\r m_\r,
\end{aligned}	
\end{equation}
known as the Sylvester map. 
In particular the divisor of the resultant detects the primes $\p\subseteq A$ for which every multigraded slice $\de_1^\nu$ with sufficiently high multidegree $\nu$ fails to be locally surjective at $\p$. 
The Sylvester map can be completed to the left to form the multigraded Koszul complex $\bK_\bullet=\bK_\bullet(\Seq,M)$. 
Another approach to the construction of the resultant is to define it as the Cayley determinant of a sufficiently high multidegree slice of $\bK_\bullet$. 
In particular the resultant detects when localizations of $\bK_\bullet$ fail to be exact. 

The results of the paper are organized as follows. 
The basic properties of filter-regular elements are derived all throughout \cref{sec:multi}, and in \cref{northcott} we verify that the Koszul complex $\bK_\bullet(\Seq,M)$ is acyclic if $\Seq$ is a filter-regular sequence. 
In \cref{div:stab} we prove that the divisor $\div_\ufd((M/(\Seq)M)_\nu$ stabilizes for $\nu$ large enough. 
The idea for proving this key stabilization property is that the multiplicity of $(M/(\Seq)M)_\nu$ at some prime should be seen as a local Hilbert function, as $\nu$ varies. 
Our approach is therefore  different than the one of R\'emond \cite[Theorem 3.3]{RemondCh5}, that uses elimination theory, and than the usual cohomological approach, see \cref{rmk:cohomology}. 
In \cref{ann=det} and \cref{res:koszul=remond} we prove that the two definitions of the multigraded M-resultant, respectively via the annihilant and via the determinant of the Koszul complex, coincide. 
In \cref{thm:main:chardin} we prove the main ``$\p$-adic'' multiplicity estimate for M-resultants, and in \cref{thm:main:roy:CM} we deduce a multiplicity estimate for the R\'emond resultant, in a form more suitable for geometrical applications. 
In \cref{sec:application} we introduce the theory of interpolation on a commutative algebraic group embedded in multiprojective space, 
we describe the primary decomposition of the so-called interpolation ideal and we discuss its relation with the surjectivity of the evaluation map. 
Finally, in \cref{thm:res:estimate} we state our main corollary, which is a lower bound on the multiplicity of the Chow form of the group at a sequence of interpolation polynomials.  

\addtocontents{toc}{\SkipTocEntry} \section*{Acknowledgements}

I would like to gratefully acknowledge my supervisor Prof.Damien Roy for his wholehearted support, his careful reading of this paper and his valuable advices. 
I also thank an anonymous referee for bibliographical suggestions.  
This work was supported in part by the full International Scholarship of the University of Ottawa and the FGPS, in part by the International Doctoral Scholarship 712230205087, and in part by NSERC. 

\tableofcontents

\section{Preliminaries on multigraded commutative algebra}
\label{sec:multi}


\subsection{Multigraded rings and modules}
\label{sec:multi:def}

Let $\N_+$ denote the set of positive integers and let $\q\in\N_+$ be given.  We say that a ring $R$ is \emph{multigraded} (or $\N^\q$-graded if $\q$ may not be clear from the context) if it admits a decomposition $R=\bigoplus_{\d\in\N^\q} R_\d$ such that  $R_{\mbf{a}}R_{\mbf{b}}\subseteq R_{\mbf{a}+\mbf{b}}$ for every $\mbf{a},\mbf{b}\in\N^\q$. 
An $R$-module $M$ is multigraded if it decomposes as $M=\bigoplus_{\d\in\N^q} M_\d$ and  $R_\mbf{a}M_\mbf{b}\subseteq M_{\mbf{a}+\mbf{b}}$ for every $\mbf{a},\mbf{b}\in\N^\q$.
Every element of $\N^\q$ is called a \emph{multidegree}. For every $p=1,\ldots,\q$ we denote by $\e_p$ the multidegree corresponding to the $p$-th canonical basis vector of $\N^\q$, i.e. such that $\e_{p,j}=\delta_{p,j}$, where $\delta$ is the Kronecker symbol. We also let $\zero$ and $\uno$ to be the elements of $\N^\q$ with all the components equal to 0 and 1 respectively. We introduce on $\N^\q$ the componentwise partial order $\leq$, such that $\d^{(1)}\leq\d^{(2)}$ if and only if $\d^{(1)}_i\leq\d^{(2)}_i$ for every $i=1,\ldots,\q$. Then we state that a property holds \emph{for $\d\in\N^\q$ large enough} if there exists $\d^{(0)}\in\N^\q$ such that the property holds for every $\d\in\N^\q$ satisfying $\d\geq\d^{(0)}$. For each $\d\in\N^\d$ we say that $M_\d$ is a multihomogeneous component of $M$ of multidegree $\d$ and we call every element of $M_\d$ a multihomogeneous element of $M$ of multidegree $\d$. 
A multigraded $R$-module $M$ is \emph{eventually zero} if $M_\d=\{0\}$ for $\d$ large enough.

\begin{remark}
In the case $\q=1$, the notions of multigraded rings and modules coincide with the more common notions of graded rings and modules. The reader interested only in the graded case can read all this article by replacing everywhere the words multigraded, multihomogeneous and multiprojective with graded, homogeneous and projective respectively. 
\end{remark}


\subsection{Multihomogeneous submodules and relevant ideals}
\label{sec:multi:relevant}

We say that an $R$-submodule $N\subseteq M$ is multihomogeneous if it is generated by multihomogeneous elements or, equivalently, if $N=\bigoplus_{\d\in\N^\q} N\cap M_\d$.
Given a multihomogeneous submodule $N$ of $M$ we have induced $R$-module structures on $N$ and $M/N$, respectively with $N_\d=N\cap M_\d$ and $(M/N)_\d=M_\d/N_\d$.
An ideal $I\subseteq R$ is multihomogeneous if it is a multihomogeneous submodule of $R$.
We say that a multihomogeneous submodule $N\subseteq M$ is \emph{irrelevant} if $N_\d = M_\d$ for $\d$ large enough or, equivalently, if $M/N$ is eventually zero. A multihomogeneous submodule $N\subseteq M$ is \emph{relevant} if it is not irrelevant. A multihomogeneous ideal $I\subseteq R$ is relevant (irrelevant) if $I$ is a relevant (irrelevant) submodule of $R$.

If $\mcl{F}\subseteq R$ is a family of multihomogeneous elements of $R$, we denote by $(\mcl{F})$ the multihomogeneous ideal generated by them. If $N\subseteq M$ is a multihomogeneous submodule of a multigraded $R$-module $M$, if $\mcl{N}\subseteq M$ is any family of multihomogeneous elements of $M$ and if $\mcl{F}$ is any family of multihomogeneous elements of $R$, then the colon submodule (module quotient) $(N:_M\mcl{F}):=\{m\in M:\, fm\in N\ \forall f\in \mcl{F}\}$ is a multihomogeneous submodule of $M$ and the colon ideal (ideal quotient) $(N:_R\mcl{N}):=\{r\in R:\, r\eta\in N\  \forall \eta\in \mcl{N}\}$ is a multihomogeneous ideal of $R$. In particular $\op{Ann}_R(M):=(0:_RM)$ is a multihomogeneous ideal.

Given an $R$-module $M$ we denote by $\Ass_R(M)$ the set of associated primes of $M$ in $R$. If $M$ is a multigraded $R$-module and $\mfk p\in\Ass_R(M)$, then $\mfk p$ is a multihomogeneous prime ideal of $R$ and is equal to $(0:_Mm)$ for some multihomogeneous element $m\in M$ \cite[Lemme 2.5]{RemondCh5}.
If $R$ is Noetherian and $M$ is a finitely generated $R$-module, then $\Ass_R(M)$ is a finite set.


\subsection{Filter-regular sequences and f-depth}
\label{sec:multi:regular}

Given a multigraded ring $R$, a multigraded $R$-module $M$ and a multihomogeneous element $\seq\in R$, we say that $\seq$ is \emph{filter-regular} for $M$ if the colon submodule $(0:_M \seq)$ is eventually zero or, equivalently, if the multiplication by $\seq $ induces injective maps $M_\nu\xrightarrow{\cdot \seq\,} M_{\nu+\d}$ for $\nu$ large enough, where $\d$ is the multidegree of $\seq$. 
In particular, if $M$ is eventually zero then  any multihomogeneous $\seq\in R$ is filter-regular for $M$.

A collection $\Seq=(\seq_0,\ld,\seq_\r)$ of multihomogeneous elements of $R$ is a \emph{filter-regular sequence} for $M$ if  $\seq_i$ is filter-regular for the module $M/(\seq_0,\ld,\seq_{i-1})M$ for $i=0,\ld,\r$.
For every multihomogeneous ideal $J$ of $R$ and every multigraded module $M$ we define $\fgr(J,M)\in\N\cup\{\infty\}$ to be the supremum of all the $\r\in\N$ such that there exists a filter regular sequence $\Seq=(\seq_0,\ld,\seq_{\r-1})$ for $M$ with $\seq_i\in J$ for  $i=0,\ld,\r-1$. 
The following fact is easy to prove. See for example \cite[Proposition 2.5]{Viet2013}.

\begin{proposition}\label{filter-regular:relevant}
Let $R$ be a Noetherian multigraded ring, $M$ a finitely generated multigraded $R$-module and $\seq\in R$ a multihomogeneous element.
Then $\seq$ is filter-regular for $M$ if and only if is not contained in any \emph{relevant} associated prime of $M$ in $R$. 
\end{proposition}

The above proposition is useful to prove the existence of filter-regular elements, expecially when coupled with the following multihomogeneous version of the Prime Avoidance lemma.

\begin{lemma}\label{prime:avoidance}
Let $R$ be a Noetherian multigraded ring, let $\p_1,\ld,\p_s$ be relevant multihomogenous primes of $R$ and let $I$ be a multihomogenous ideal of $R$ with $I\not\subset\p_i$ for $i=1,\ld,s$. Then for every $\nu\in\N^\q$ large enough there exists $\seq\in I$ multihomogeneous of multidegree $\nu$ such that $\seq\not\in\p_i$ for $i=1,\ld,s$.
\end{lemma}

\begin{proof}
We may assume there are no inclusions among the $\p_i$. 
For $i=1,\ld,s$ let $J_i:={I\prod_{j\neq i}\p_j}$.
It is well known that if a prime ideal contains the product of some ideals, then it must contain one of them. Moreover, the product of multihomogenoeus ideals is multihomogeneous.
Therefore, there is $x_i\in J_i$ multihomogeneous of multidegree, say, $\d^{(i)}$, such that $x_i\not\in\p_i$. 
Since $\p_i$ is relevant and $R$ is Noetherian, we have that for all $\nu\in\N^\q$ large enough there is $y_i\in R_{\nu-\d^{(i)}}$ with $y_i\not\in\p_i$.
Then $f=\sum_{i=1}^s x_iy_i$ has the required property.
\end{proof}


\begin{remark}
Filter-regular sequences are related to superficial sequences and (mixed) multiplicity systems, and are widely used in the study of Rees algebras and Hilbert functions of local rings. See for example \cite{Trung2010}, \cite{Rossi2007}, \cite{Viet2012} or \cite{Kirby}.
\end{remark}


\subsection{Multigraded polynomial rings and componentwise free modules}
\label{sec:multi:poly}

Given an integer $\q\in\N_+$ as in \cref{sec:multi:def} and positive natural numbers $n_1,\ldots,n_\q\in \N_+$, we introduce the set of variables $\x=(x_{p,i})_{p=1,\ld,\q,\, i=0,\ld, n_p}$ and for every $p=1,\ld,\q$ we denote by $\x_p$ the subcollection $\x_p=(x_{p,0},\ld,x_{p,n_p})$.
If $\ufd$ is any ring, we denote by $\ufd[\x]$ the polynomial ring with coefficients in $\ufd$ and variables in $\x$. We consider on $\ufd[\x]$ the unique $\N^\q$-graded ring structure such that every nonzero constant $a\in \ufd$ has multidegree $\zero$ and that $x_{p,i}$ is multihomogeneous of multidegree $\e_p$, for every $p=1,\ld,\q$ and $i=0,\ld, n_p$.
We define a \emph{componentwise free $\ufd[\x]$-module} to be a finitely generated multigraded $\ufd[\x]$-module $M$ whose multihomogeneous components $M_\d$ are free $\ufd$-modules of finite ranks.
%


\subsection{The Hilbert polynomial and the relevant dimension}
\label{sec:multi:hil}

Given a ring $\ufd$, we denote by $Mod_\ufd$ the category of finitely generated $\ufd$-modules. An \emph{additive integer-valued function} on $Mod_\ufd$ is a mapping $\lambda:\, Mod_\ufd\to\Z$ satisfying $\lambda(M)=\lambda(M')+\lambda(M'')$ for every short exact sequence $0\to M'\to M\to M''\to 0$ in $Mod_\ufd$.
If $\field$ is a field, $R$ is an Artinian ring and $\ufd$ is an integral domain with field of fractions $\field$, then the dimension $\dim_\field(-)$, the length $\ell(-)$ and the \emph{generic rank} $\op{rank}_\ufd(-)=\dim_\field(-\otimes_\ufd\field)$ are additive integer valued functions respectively on $Mod_\field$, $Mod_R$ and $Mod_\ufd$.

If $\ufd$ is a Noetherian ring and $M$ is a finitely generated multigraded $\ufd[\x]$-module, then every multihomogeneous component $M_\d$ is a finitely generated $\ufd$-module.
If  $\lambda$ is an additive integer-valued function on $Mod_\ufd$, we introduce the \emph{Hilbert function} $h_{M,\lambda}:\N^\q\to \Z$ given by $h_{M,\lambda}(\d)=\lambda(M_\d)$.

\begin{proposition}\label{hilbert:polynomial}
Let $\ufd$ be a field, an Artinian ring or a Noetherian integral domain, and let $\lambda(-)$ be $\dim_\ufd(-)$, $\ell(-)$ or $\op{rank}_\ufd(-)$ respectively, as above.
Then for every finitely generated $\ufd[\x]$-module $M$ there is a unique polynomial $P_{M,\lambda}$ in $\q$ variables and with coefficients in $\Q$, called the \emph{Hilbert polynomial}, such that $h_{M,\lambda}(\d)=P_{M,\lambda}(\d)$ for every sufficiently large $\d\in\N^\q$. 
\end{proposition}
\begin{proof}
The case of an Artinian ring includes the case of a field, which in turn implies the case of an integral domain. The standard reference is \cite{VdW1929}, although it actually covers only the bigraded case over a field. For a modern and more complete treatment of the field case see \cite[Theorem 2.10]{RemondCh5} or \cite[Lemma 2.8]{Maclagan2003}. For the Artinian ring case see \cite[Theorem 2.6]{Trung2010} or \cite{Herrmann1997}.
\end{proof}

In case $\ufd$ is a Noetherian integral domain with fraction field $\field$ we also use the notation $H_M:=P_{M,\op{rank}_{\ufd}}=P_{M\otimes_{\ufd[\x]}\field[\x],\dim_{\field}}$ for brevity. In case $H_M\equiv 0$ (i.e. if and only if $M\otimes_{\ufd[\x]}\field[\x]$ is eventually zero) we set $\rdim_\ufd(M):=-1$. Otherwise, we denote the total degree of $H_M$ by $\rdim_\ufd(M)$ and we call it the \emph{relevant dimension} of $M$. 
If $\rdim_\ufd(M)=0$ or $\rdim_\ufd(M)=-1$ the Hilbert polynomial $H_M$ is a constant nonnegative integer, and we define the \emph{relevant degree} $\rdeg_\ufd(M)\in\N$ to be this integer.

%
%
%


\subsection{Multiprojective subschemes and multisaturation}
\label{sec:multi:geometry}

Given a field $\kk$ and $n\in\N_+$ we denote by $\PP^n_\kk=\op{Proj} (\kk[X_0,\ld,X_n])$ the projective space of dimension $n$ over $\kk$. Given an integer $\q\in\N_+$ as above and a collection $\n=(n_1,\ldots,n_\q)\in\N_+^\q$ of positive natural numbers, we define $\PP^\n_\kk:=\PP^{n_1}_\kk\times\dots\times\PP^{n_\q}_\kk$ and we call it a \emph{multiprojective space}. It is a reduced irreducible scheme over $\op{Spec} \kk$ of dimension $\abs{\n}:=n_1+\ldots+n_\q$. Following \cite[Section 2.5]{RemondCh5}%
, we see that its underlying (Zariski) topological space is naturally set-theoretically in bijection with the set of relevant multihomogeneous prime ideals of $\kk[\x]$. In fact, to every closed subscheme $Z$ of $\PP^\n_\kk$, which we call a \emph{multiprojective subscheme}, is attached a multihomogeneous ideal $I\subseteq \kk[\x]$, called the \emph{ideal of definition} of $Z$ and denoted by $\mcl{I}(Z)$. Conversely, every multihomogeneous ideal $I\subseteq \kk[\x]$ defines a multiprojective subscheme $\mcl{Z}(I)$ such that $\mcl{Z}(\mcl{I}(Z))=Z$ for every multiprojective subscheme $Z$ of $\PP^\n_\kk$. 
%
%
%
For every multihomogeneous ideal $I\subseteq\kk[\x]$ we define its \emph{multisaturation} by $\wb{I}:=\mcl{I}(\mcl{Z}(I))$, so that the ideals in the image of $Z\mapsto \mcl{I}(Z)$ are those satisfying $I=\wb{I}$. 

\begin{proposition}\label{multisaturation}
	The following are equivalent definitions for the multisaturation of $I$.
	\begin{enumerate}[(i)]
	\item 
	$\wb{I}=\{f\in\kk[\x]:\exists \d_f\in\N^\q\ f\kk[\x]_{\d_f}\subseteq I\}$.
	\item 
	$\wb{I}$ is maximal among all multihomogeneous ideals $J$ such that $J_\d=I_\d$ for $\d$ large enough.
	\item
	$\wb{I}$ is the intersection of the primary ideals of $\kk[\x]$ appearing in a minimal primary decomposition of $I$ and corresponding to relevant primes.
	\end{enumerate}
\end{proposition}

\begin{proof}
(i) is proved in \cite[Proposition 2.17]{RemondCh5}. 
For (ii), the inclusion $I\subseteq\wb{I}$ is clear. 
$\wb{I}$ is generated by finitely many multihomogeneous elements $f_1,\ld,f_\r$ and by (i) there are $\d_{f_1},\ld,\d_{f_\r}\in\N^\q$ such that $f_i\kk[\x]_{\d_{f_i}}\subseteq I$. 
If $\d_1$ is an upperbound for the multidegrees of $f_1,\ld,f_\r$ and if $\d_2$ is an upperbound for $\d_{f_1},\ld,\d_{f_r}$, then for every $\d\geq\d_1+\d_2$ we have $\wb{I}_\d=I_\d$. 
Moreover, if $J$ is a multihomogeneous ideal of $\kk[\x]$ such that $J_\d=I_\d$ for $\d$ large enough, then $J\subseteq \wb{I}$ by (i).
Finally, (iii) is a consequence of \cite[Lemme 2.4]{RemondCh5}%
\footnote{By (i), our $\wb{I}$ coincides with the characteristic ideal $\mfk U_\emptyset(I)$ of R\'emond.}. 
\end{proof}


\subsection{More on the relevant dimension}
\label{sec:multi:dim}

Given a Noetherian integral domain $\ufd$ with fraction field $\field$ and a finitely generated $\ufd[\x]$-module $M$ we defined the relevant dimension $\rdim_\ufd(M)$ in terms of the total degree of the Hilbert polynomial $P_{M,\op{rank}_\ufd}$. 
We denote by $\dim Z$ the dimension of a multiprojective subscheme $Z$ of $\mbb{P}^\n_\field$, and for a prime ideal $\p\subseteq\ufd[\x]$ we let $\dim(M_\p)$ be the Krull dimension of the module $M_\p$, defined in terms of chains of prime ideals of $\ufd[\x]_\p$ containing the annihilator of  $M_\p$ \cite[Appendix]{Cohen-Macaulay}. Then we define
\[
     e(M) := \max \{\abs{\n}-\op{ht}(\p):\ \p\subseteq\field[\x]\text{ relevant prime}, \Ann_{\field[\x]}(M\otimes_\ufd \field)\subseteq \p\},
\]
where $\op{ht}(\p)$ denotes the height of $\p$.

\begin{proposition}\label{dim=dim}
Let $\ufd$ be a Noetherian integral domain with fraction field $\field$ and $M$ a finitely generated $\ufd[\x]$-module. Then
\[  \rdim_\ufd(M)=\rdim_\field(M\otimes_\ufd\field)=\dim\mcl{Z}(\Ann_{\field[\x]}(M\otimes_\ufd \field))=e(M)=\max_\p \dim(M_\p), \]
where in the rightmost formula $\mfk{p}$ ranges through the relevant multihomogeneous primes of $\ufd[\x]$ such that $\p\cap\ufd =(0)$.
\end{proposition}
\begin{proof}
 The first equality is clear, the second and the third are essentially proved in \cite[Theorem 2.10, Section 2.5]{RemondCh5}, the last follows from the fact that the primes of $\field[\x]$ are in bijection with the primes of $\ufd[\x]$ such that $\p\cap\ufd =(0)$.
\end{proof}

The next lemma shows that the operation of quotienting by a filter-regular sequence has the effect of decreasing the total degree of the Hilbert polynomial, by an amount at least equal to the lenght of the sequence.

\begin{lemma}\label{filter-regular:hilbert}
Let $R$ be an Artinian ring and $M$ a finitely generated multigraded $R[\x]$-module.
Let $\seq\in R[\x]$ be a filter-regular element of multihomogeneous degree $\d$ for $M$ and $\lambda$ be the length function on $Mod_R$.
If $P_{M,\lambda}$ is not the zero polynomial, then the total degree of $P_{M/\seq M,\lambda}$ is at least one less the total degree of $P_{M,\lambda}$.
If $\d\geq \uno$ then this inequality is indeed an equality.
\end{lemma}
\begin{proof}
Since  $\seq\in R[\x]$ is filter-regular for $M$ we have a short exact sequence $0\to M_\nu\to M_{\nu+\d} \to (M/\seq M)_{\nu+\d}\to 0$ for $\nu$ large enough. 
From the additivity of $\lambda$ we get $P_{M/\seq M,\lambda}(\nu+\d)=P_{M,\lambda}(\nu+\d)-P_{M,\lambda}(\nu)$ for $\nu$ large enough, which implies the first statement by inspection.
The second part is similar, and uses the fact that the coefficients of a Hilbert polynomial corresponding to monomials of highest total degree are nonnegative \cite[Theorem 2.6]{Trung2010}.
\end{proof}

\begin{corollary}\label{filter-regular:hilbert:rdim}
Let $\ufd$ be a Noetherian integral domain, let $M$ be a finitely generated $\ufd[\x]$-module and let $J$ be a multihomogeneous ideal of $\ufd[\x]$.
Then
\[
\rdim_\ufd(M/J M)\leq\max\left\{-1,\,\rdim_\ufd(M)-\fgr(J,M)\right\}.
\]
\end{corollary}

\begin{remark}
To see that the hypothesis $\d\geq\uno$ in \cref{filter-regular:hilbert} is necessary, take $\q=2$, $n_1=2$, $n_2=1$, $M=\ufd[\x]/(x_{2,1})$ and $\seq=x_{2,0}$, for which $\rdim_\ufd(M)=2$ and $\rdim_\ufd(M/\seq M)=-1$.
However, it is often possible to weaken this condition: see \cite[Theorem 2.10, (3)]{RemondCh5}.
\end{remark}


\section{Koszul complexes and resultants}
\label{sec:res}

%
\subsection{Multigraded Koszul complexes}
\label{sec:res:koszul}

Given a commutative ring $R$, an $R$-module $M$ and a sequence $\Seq=(\seq_0,\ld,\seq_r)$ of elements of $R$, the Koszul complex $\bK_\bullet(\Seq,M)$ is a finite complex of $R$-modules given by 
\[
\bK_\bullet(\Seq,M) \ :=\ 0\to (\BigWedge^{\r+1} L)\otimes M\xrightarrow{\de_{\r+1}} \dots \xrightarrow{\ \de_2\ }(\BigWedge^1 L)\otimes M \xrightarrow{\ \de_1\ } M\to 0 ,
\]
where $L$ is the free $R$-module $R^{\r+1}$ equipped with a basis $(e_0,\ldots,e_\r)$, the tensor products are taken over $R$, and the differentials $\de_p$ are defined by
\[
\de_p(e_{i_1}\wedge\dots\wedge e_{i_p}\otimes m) = 
\sum_{s=1}^p (-1)^{s+1} \seq_{i_s} e_{i_1}\wedge \dots \wedge \wh{e_{i_s}}\wedge \dots \wedge e_{i_p}\otimes m.
\]
The homology modules $\oH_p(\bK_\bullet(\Seq,M))$ are denoted by $\oH_p(\Seq,M)$ for short and their direct sum $\oH_\bullet(\Seq,M)$ is called the \emph{Koszul homology} of the sequence $\Seq$ with coefficients in $M$. For the $0$-th and $(\r+1)$-th homology modules we have the natural isomorphisms $\oH_0(\Seq,M)\cong M/(\Seq)M$, $\oH_{r+1}(\Seq,M)\cong (0:_M (\Seq))$. Moreover, the annihilator $\Ann_R(\oH_\bullet(\Seq,M))$ contains both $\Ann_R(M)$ and the ideal $(\Seq)$.
We refer to Section 1.6 of \cite{Cohen-Macaulay} for more on the general theory of Koszul complexes.

Suppose now that $R$ is multigraded as in \cref{sec:multi:def}, $M$ is a multigraded $R$-module, $\dd=(\dz,\ld,\dr)$ is a collection of nonzero multidegrees and $\Seq=(\seq_0,\ld,\seq_\r)$ is a sequence of multihomogeneous elements of $R$ with multidegrees prescribed by $\dd$. Then we can introduce on the $R$-modules $\mbf{K}_p(\Seq,M)=\left(\BigWedge^p L\right)\otimes_{R} M$ the natural $\N^\q$-grading for which $\deg_{\N^\q}(e_{i_1}\wedge\dots\wedge e_{i_p}\otimes m)=\d^{(i_1)}+\dots + \d^{(i_p)}+\deg_{\N^\q}(m)$, for $m$ multihomogeneous. This is also done in \cite[Section 3]{Viet2012} and is similar to the homogeneous case \cite[Remark 1.6.15]{Cohen-Macaulay} \cite{Chardin1993}.
We notice that the differentials preserve this grading, so that the homology modules inherit a multigraded structure.
We then write $\bK^\nu_\bullet(\Seq,M)$ and $\oH^\nu_\bullet(\Seq,M)$ for the component of  multidegree $\nu$ respectively of the Koszul complex and of the Koszul homology.
If we denote the restricted differentials by 
\[
	\de_p^\nu:\bK_p^\nu(\Seq,M)\to\bK_{p-1}^\nu(\Seq,M)
\]
then for every $\nu\in\N^\q$ we see that $\bK^\nu_\bullet(\Seq,M)$ is a complex of $R_\zero$-modules with differentials $\de^\nu_p$ and homology $\oH^\nu_\bullet(\Seq,M)$.

The next proposition is an adaptation to filter-regular sequences of a classical result that relates the existence of regular sequences to the vanishing of higher Koszul homology. 
We give a proof along the lines of \cite[Section 8.5, Theorem 6]{Northcott}, that uses the following definition. 

\begin{definition}\label{def:northcott}
Let $R$ be a Noetherian multigraded ring, $\Seq=(\seq_0,\ld,\seq_{s-1})$ a sequence of $s$ multihomogeneous elements of $R$, and $M$ a finitely generated multigraded $R$-module.
If there is at least one integer $\lambda\in\{1,\ld,s\}$ such that $\oH_\lambda(\Seq,M)$ is not eventually zero, we define $\lamb M$ to be the largest such integer.
Otherwise, we set $\lamb M:={{-\infty}}$.
\end{definition}

\begin{proposition}\label{northcott}
Let $R$, $M$ and $\Seq$ be as in \cref{def:northcott}, and let $J=(\Seq)$. We have:
\begin{enumerate}[(i)]
\item
if $\beta\in J$ is filter-regular for $M$, then $\lamb {M/\beta M} = \lamb M +1$, where we let ${-\infty}+1:={-\infty}$;
\item
$\fgr(J,M)=s-\lamb M$, with $\fgr(J,M)$ as in  \cref{sec:multi:regular} and $s-({{-\infty}}):=\infty$; 
\item
if $\Seq$ is a filter-regular sequence for $M$, then $\bK^\nu_\bullet(\Seq,M)$ is acyclic (i.e. its $p$-th homology modules vanish for $p\geq 1$) for $\nu$ large enough.
\end{enumerate}
\end{proposition}

\begin{proof}
Let $\beta\in J$ be filter-regular for $M$. By definition, $\beta$ is a multihomogeneous element of $R$. Let $\d\in\N^\q$ be its multidegree.
The $R$-module $M/\beta M$ is finitely generated and multigraded, so $\lamb  {M/\beta M}$ is defined. 
For every $\nu\in\N^\q$ large enough we have an exact sequence 
\[
0\to M_{\nu-\d}\xrightarrow{\beta}M_\nu \to (M/\beta M)_\nu\to 0,
\]
where the first map is induced by the multiplication by $\beta$ in $M$. The collection of these maps induce a long exact sequence in Koszul homology that at the level of multihomogeneous components takes the form
\[
\to
\oH_\mu(\Seq,M)_{\nu-\d}
\xrightarrow{\beta}
\oH_\mu(\Seq,M)_{\nu}
\to
\oH_\mu(\Seq,M/\beta M)_{\nu}
\to
\oH_{\mu-1}(\Seq,M)_{\nu-\d}
\to
\]
 Since $\oH_\mu(\Seq,M)$ is annihilated by all elements of $J$, the above exact sequence simplifies to 
\[
0
\to 
\oH_\mu(\Seq,M)_{\nu}
\to
\oH_\mu(\Seq,M/\beta M)_{\nu}
\to
\oH_{\mu-1}(\Seq,M)_{\nu-\d}
\to 
0
\]
For $\mu\great \lamb M+1$ both $\oH_\mu(\Seq,M)$ and $\oH_{\mu-1}(\Seq,M)$ are eventually zero modules and hence we obtain $\oH_\mu(\Seq,M/\beta M)_{\nu}=0$ as well for sufficiently large $\nu$. In particular, if $\lamb M={-\infty}$ we have $\lamb{M/\beta M}={-\infty}$ as well.
On the other hand if $\lamb M\geq 0$ and $\mu=\lamb M+1$, we obtain an isomorphism
\[
\oH_\mu(\Seq,M/\beta M)_{\nu}
\cong
\oH_{\mu-1}(\Seq,M)_{\nu-\d}
\]
for sufficiently large $\nu\in\N^\q$, which shows that $\oH_\mu(\Seq,M/\beta M)$ is not eventually zero. 
Therefore, we have $\lamb {M/\beta M}=\lamb M+1$, which is (i).


To prove (ii), first suppose that $\fgr(J,M)=0$. This means that no element of $J$ is filter-regular for $M$. In this case all multihomogeneous elements of $J$ are contained in the union of the relevant associated primes of $M$ by \cref{filter-regular:relevant}, and so by \cref{prime:avoidance} all of $J$ is contained in one of them, say $\mfk{p}$.
Write $\mfk{p}=(0:_Rm)$ for a multihomogeneous element $m\in M$.
Since $\mfk{p}$ is a relevant prime, $Rm$ is a module which is not eventually zero and is contained into the colon module $(0:_M\mfk{p})$, which in turn is contained in $(0:_MJ)$. This proves that $\oH_s(\Seq,M)=(0:_M J)$ is not eventually zero and so $\lamb M=s$.

Now assume that $\fgr(J,M)\great 0$ and $\fgr(J,M)\neq \infty$. Then by definition there exists $\beta\in J$ that is filter-regular for $M$ and $\fgr(J,M/\beta M) = \fgr(J,M)-1$.
Then we have, by induction on $\fgr(J,M)$ and (i) above, that 
\[
\fgr(J,M)=\fgr(J,M/\beta M) +1 = s +1 - \lamb {M/\beta M} = s-\lamb M.
\]
On the other hand, if $\fgr(J,M) = \infty$, we can find a filter-regular sequence $\underline{\beta}=(\beta_1,\ld,\beta_n)$ for $M$, with $n$ arbitrarily large. 
Let $N=M/(\underline\beta) M$. Then by repeatedly using (i) we get $\lamb N = \lamb M+n$. However, we clearly have $\lamb N \leq s$, so we get a contradiction if $\lamb M\geq 0$ and $n\great s$. 

Finally, suppose that $\Seq$ is a filter-regular sequence for $M$. Then $\fgr(J,M)\geq s$ and so, by (ii) above, we get $\lamb M\leq 0$. 
This exactly means that $\bK^\nu_\bullet(\Seq,M)$ is acyclic for $\nu$ large enough.
\end{proof}

\begin{remark}\label{northcott:alternative}
	Another approach to prove \cref{northcott} is to to use the fact that a multihomogeneous element $\seq\in R$ is filter-regular for $M$ if and only if it is regular for $M_{\geq \d}=\bigoplus_{\d'\geq\d} M_{\d'}$ for some $\d\in\N^\q$.
\end{remark}

\begin{remark}\label{rmk:northcott}
\Cref{northcott} also shows that all maximal filter-regular sequences for $M$ in $J$ have the same number of elements.
\end{remark}
%
%


\subsection{Contents and divisors of torsion modules}
\label{sec:res:ann}

Given a Noetherian integral domain $\ufd$ and a finitely generated $\ufd$-module $M$, we say that $M$ is a \emph{torsion module} if $\Ann_\ufd(M)\neq 0$.
If $\mfk{p}$ is a prime ideal of $\ufd$, then the localization $M_\mfk{p}$ is not the zero module if and only if $\Ann_\ufd(M)\subseteq \mfk{p}$. 
In particular, choosing $\mfk{p}=\{0\}$, we see that $M$ is a torsion $\ufd$-module if and only if $M\otimes_\ufd\field=0$, where $\field$ is the field of fractions of $\ufd$.
Moreover, if $\mfk{p}$ is a prime ideal of height 1 and $M$ is a torsion $\ufd$-module, then $M_\mfk{p}$ is a torsion $\ufd_{\mfk{p}}$-module and thus it has finite length $\ell(M_{\mfk{p}})$. This length is nonzero if and only if $\mfk{p}$ is a minimal associated prime of $M$ in $\ufd$.

\begin{definition}\label{def:div}
If $\ufd$ is a Noetherian integral domain, we denote by $\op{Div}(\ufd)$ the free abelian group generated by the primes $\mfk{p}$ of $\ufd$ of height 1. If $M$ is a torsion $\ufd$-module we define $\div_\ufd(M)\in\op{Div}(\ufd)$ by
\[
		\div_\ufd(M):= \sum_{\mfk{p}} \ell(M_{\mfk{p}}) [\mfk{p}] ,
\]
where the sum ranges over all primes $\mfk{p}$ of $\ufd$ of height 1. If $M$ is not torsion, we define $\div(M)=0$.
\end{definition}

We refer to \cite{bourbaki}[Chap. 7, par.4] for the theory of divisors of torsion modules. 
In case $\ufd$ is an UFD ring, every prime of height 1 is principal, generated by an irreducible (prime) element $\pi\in\ufd$, well defined up to multiplication by a unit $u\in\ufd^\times$.

\begin{definition}\label{def:ann}
If $\ufd$ is an UFD ring, $M$ is a torsion $\ufd$-module and $\op{irr}(\ufd)$ is a choice of representatives for the irreducible elements of $\ufd$, we define the \emph{content} $\chi_\ufd(M)\in \ufd$ of $M$ by the formula
\[
 \chi_\ufd(M) := \prod_{\pi\in \op{irr}(\ufd)} \pi^{\ell(M_{(\pi)})} .
\]
\end{definition}

In elimination theory, the content of a torsion module is sometimes called \emph{annihilant form} \cite[Definition 1.22]{DAndrea2000}. 
      This notion is also related to the MacRae invariants and the zeroth Fitting ideals. 

The following is a technical result that we will use in the next paragraph to be able to define the resultant.
Together with \cref{filter-regular:generic} below, it generalizes \cite[Theorem 3.3]{RemondCh5}, but our proof is considerably different, since we cannot make use of multihomogeneous elimination theory here.
Instead, we make the key observation that the multiplicities appearing in the divisors under consideration can be computed as local Hilbert functions.
Then to prove that they are eventually constant, it suffices to show that the corresponding Hilbert polynomials have degree zero.


\begin{proposition}\label{div:stab}
Let $\ufd$ be a Noetherian integral domain and $M$ a finitely generated multigraded $\ufd[\x]$-module that is projective as an $\ufd$-module. 
Let also $\rdim_\ufd(M)=\r$ and $\Seq=(\seq_0,\ld,\seq_\r)$ be a filter-regular sequence for $M$ in $\ufd[\x]$.
Then there is $\nu_0\in\N^\q$ such that $\div_\ufd((M/(\Seq)M)_\nu)=\div_\ufd((M/(\Seq)M)_{\nu_0})\neq 0$ for every $\nu\geq \nu_0$.
\end{proposition}

\begin{proof}
Let $N=M/(\Seq)M$ and let $\field$ be the fraction field of $\ufd$. Since $\Seq$ is a filter-regular sequence for $M$ of length $\rdim_\ufd(M)+1$ we see that $\rdim_\ufd(N)=-1$ by \cref{filter-regular:hilbert:rdim}.
This implies that $(N_\nu)\otimes_\ufd\field=0$ for $\nu\in\N^\q$ large enough, which is equivalent to say that $N_\nu$ is a torsion $\ufd$-module, or that $\Ann_\ufd(N_\nu)\neq 0$.

We now show that the ideal $\Ann_\ufd(N_\nu)$ is constant for $\nu$ large enough. 
Indeed $M$ is generated as an $\ufd[\x]$-module by finitely many elements with multidegrees bounded above by some $\nu_1\in\N^\q$. For $\nu_1\leq\nu\leq\nu'$ we have $\Ann_\ufd(N_\nu)\subseteq \Ann_\ufd(N_{\nu'})$ and so we conclude by the noetherianity of $\ufd$.
The discussion preceding \cref{def:div} shows that a prime $\mfk{p}$ of height 1 appears in $\div_\ufd(N_\nu)$ if and only if $\mfk{p}\supseteq \Ann_\ufd(N_\nu)$. Since the latter is constant for $\nu$ large enough, we deduce that also the prime ideals appearing in $\div_\ufd(N_\nu)$ form a fixed finite set for $\nu$ large enough.

Let $\p$ be such a prime and let $(-)_\p$ denote the localization at that prime. We will show that the number $\ell((N_\nu)_\p)$ is fixed for $\nu$ large enough.
Let $\pi$ be any nonzero element of $(\Ann_\ufd(N_\nu))_\p=\Ann_{\ufd_\p}((N_\nu)_\p)\subseteq \p\ufd_\p$ and let $\Lp:=M_\p/(\pi)M_\p$.
Since the sequence $\Seq$ is filter-regular for $M$, 
we deduce by \cref{northcott} (iii) that $\oH^\nu_i(\Seq,M)=0$ for $i=1,\ld,\r+1$ and $\nu$ large enough.
Since $\ufd_\p$ is a flat $\ufd$-module, if we apply  the localization functor $(-)_\p=(-)\times_\ufd \ufd_\p$ to an exact sequence of $\ufd$-modules (i.e. with trivial homology) we get an exact sequence of $\ufd$-modules or, alternatively, of $\ufd_\p$-modules. 
Therefore $\oH_i^\nu(\Seq,M\otimes_\ufd\ufd_\p)=0$ for $i=1,\ld,\r+1$ and $\nu$ large enough, where we still denote by $\Seq$ the induced sequence of elements in $\ufd_\p[\x]$. 
We have that $M_\p$ is a finitely generated $\ufd_\p[\x]$-module and $\ufd_\p$ is a Noetherian integral domain with $\field$ as its fraction field.  
Therefore $\rdim_{\ufd_\p}(M_\p)=\rdim_{\field}(M\otimes_\ufd\field)=\rdim_{\ufd}(M)=\r$ by \cref{dim=dim}.
Moreover, each multihomogeneous component of $M$, being a direct summand of $M$, is a (finitely generated) projective $\ufd$-module. 
Therefore every multihomogeneous component of $M_\p$ is a free $\ufd_\p$-module of finite rank.
In other words, $M_\p$ is a componentwise free $\ufd_\p[\x]$-module.
Since $\pi$ is nonzero in $\ufd_\p$, we have a short exact sequence $0\to M_\p \xrightarrow{\alpha} M_\p \xrightarrow{\beta} \Lp\to 0$,
where $\alpha$ is induced by the multiplication by $\pi$ and $\beta$ is the canonical projection.
This short exact sequence induces a long exact sequence (of $\ufd_\p[\x]$-modules) in Koszul homology, which at the level of multihomogeneous components reads as
\[
	\cdots 
	\to
	\oH_i^\nu(\Seq,M_\p)
	\to
	\oH_i^\nu(\Seq,\Lp)
	\to
	\oH_{i-1}^\nu(\Seq,M_\p)	
	\to
	\cdots,
\]
from which we deduce that $\oH_i^\nu(\Seq,\Lp)=0$ for $i=2,\ld,\r+1$ and $\nu$ large enough.
In other words, by \cref{def:northcott}, we have $\lamb {\Lp}\leq 1$.
Therefore, by \cref{northcott} (ii), we have $\fgr(\Seq, \Lp)\geq \r$, which means there exists a sequence $\underline{g}=(g_0,\ld,g_{\r-1})$ of multihomogeneous elements of $\ufd_\p[\x]$ contained in the ideal $(\Seq)$ which is filter-regular for the multigraded module $\Lp$. 
Let now $\kk=\ufd_\p/(\pi)\ufd_\p$, which is an Artinian ring, because $\p \ufd_\p$ is a prime of height 1 in the integral domain $\ufd_\p$ and $\pi$ is a nonzero element of $\p\ufd_\p$.
Let $p:\ufd_\p\to \kk$ be the natural projection and let $\lambda$ be the length function on $Mod_\kk$ as in \cref{sec:multi:hil}. 
We notice that $\Lp$ has a $\kk[\x]$-module structure that induces the $\ufd_\p[\x]$-module structure.
In particular, the sequence $p(\underline{g})$ is still a filter-regular sequence of length $\r$ for $\Lp$.
Moreover it's easy to see that $\Lp$ is a componentwise free $\kk[\x]$-module and that it satisfies the equality $\lambda((\Lp)_\nu)=\lambda(\kk)\cdot\op{rank}_{\ufd_\p}((M_\p)_\nu)$, from which we deduce that the Hilbert polynomial $P_{\Lp,\lambda}$ has degree $\r$.
Then, a repeated use of \cref{filter-regular:hilbert} shows that  $P_{\Lp/(p(\underline{g}))\Lp,\lambda}$ has degree at most zero, and so is eventually constant.
Since $M_\p/(\pi,\Seq)M_\p$ is a quotient of $\Lp/(p(\underline{g}))\Lp$, we have that the Hilbert function  $\d\mapsto\lambda((M_\p/(\pi,\Seq)M_\p)_\d)$ is constant as well, for  $\d$ large enough.
By our choice of $\pi$, for $\nu$ large enough the $\ufd$-module $(M_\p/(\pi,\Seq)M_\p)_\nu$ is nothing but $(N_\p)_\nu$, and its lenght is the same whether we consider it as a $\kk$-module or as an $\ufd_\p$-module. 
Therefore we deduce that $\ell((N_\nu)_\p)$ is constant for $\nu$ large enough.
\end{proof}

\subsection{Cayley determinants and resultants} 
\label{sec:res:det}

Let $\ufd$ be a Noetherian integral domain with fraction field $\field$ and let $\bC_\bullet$ be a finite complex of $\ufd$-modules
\[
	 0
	\to
	C_s
	\xrightarrow{d_s}
	\dots
	\xrightarrow{d_1}
	C_0
	\to
	0 .
\]
We say that $\bC_\bullet$ is \emph{generically exact} if the complex $\bC_\bullet\otimes_\ufd \field$ is an exact sequence of $\field$-vector spaces or, equivalently, if all the homology modules of $\bC_\bullet$ are torsion $\ufd$-modules. 
If $\bC_\bullet$ is a finite generically exact complex of free $\ufd$-modules of finite rank and $\{\ub_i\}_{0\leq i\leq s}$ is a system of $\ufd$-bases for the modules $C_i$, we can find a partition $\ub_i=\ub'_i\cup\ub''_i$, with $\ub''_0=\ub'_s=\emptyset$, inducing a decomposition $C_i=C'_i\oplus C''_i$, such that the matrix representations of the differentials $d_i$ take the form $\left(\begin{smallmatrix} a_i & \phi_i\\ b_i & c_i \end{smallmatrix}\right)$, where the $\phi_i$ are square matrices with nonzero determinant.
Then the \emph{Cayley determinant} of the complex $\bC_\bullet$ with respect to the above choices of $\ufd$-bases and partitions is the element of $\field^\times$ given by $\prod_{i=1}^s \op{det}(\phi_i)^{(-1)^{i+1}}$. 
It can be shown that another choice of $\ufd$-bases and partitions changes this value by multiplication with an invertible element of $\ufd$. Therefore, we can define unambiguously an element $\det_\ufd(\bC_\bullet)\in\field^\times\slash \ufd^\times$, which we still call the determinant of $\bC_\bullet$. For more on the Cayley determinant, see \cite[Appendix A]{GKZ94}.

\begin{proposition}\label{ann=det}
Let $\ufd$ be a Noetherian UFD ring, $M$ a componentwise free $\ufd[\x]$-module with $\rdim_\ufd(M)=r$ and $\Seq=(\seq_0,\ld,\seq_r)$ a filter-regular sequence for $M$.
Then $\bK_\bullet^\nu(\Seq,M)$ is generically exact for $\nu$ large enough and 
\[
		\det_\ufd(\bK_\bullet^\nu(\Seq,M))
		=
		\chi_\ufd((M/(\Seq)M)_\nu)
		\ \ (\bmod\, \ufd^\times)
\]
for $\nu\in\N^\q$ large enough.
\end{proposition}

\begin{proof}
Let $N=M/(\Seq)M$. Since $\Seq$ is a filter-regular sequence for $M$ of length $\rdim_\ufd(M)+1$, we see that $\rdim_\ufd(N)=-1$ by \cref{filter-regular:hilbert:rdim}, so $N_\nu\otimes\field=0$ and $N_\nu=\oH^\nu_0(\Seq,M)$ is torsion, if $\nu$ is large enough. 
Moreover, by \cref{northcott} (iii) we see that $\oH_p^\nu(\Seq,M)=0$ for all $\nu$ large enough and all $p\geq 1$. 
Therefore $\bK_\bullet^\nu(\Seq,M)$ is generically exact for $\nu$ large enough, and so we can consider $\det_\ufd(\bK_\bullet^\nu(\Seq,M))$.
Let $D_\nu$ any element of $\field$ representing it, and denote by $\op{ord}_\pi:\field\to\Z\cup\{\infty\}$  the valuation associated to any prime element $\pi\in\ufd$.
The thesis then amounts to proving that, for $\nu$ large enough, $\op{ord}_\pi(D_\nu)=\op{ord}_\pi(\chi_\ufd(N_\nu))$ for every prime element $\pi$ of $\ufd$.
However, the right-hand side equals $\ell((N_\nu)_{(\pi)})$ by definition, whereas the left-hand side equals $\sum_i(-1)^i\ell(\oH_i^\nu(\Seq,M)_{(\pi)})$ by \cite[Theorem 30, Appendix A, p.493]{GKZ94} (cfr. also \cite[Proposition 2]{Chardin1993}).
Since $N_\nu=\oH^\nu_0(\Seq,M)$ and $\oH_p^\nu(\Seq,M)=0$ for all $\nu$ large enough and all $p\geq 1$, the thesis follows.
\end{proof}

We now remark that \cref{div:stab} and \cref{ann=det} together imply that there exists $\nu_0\in\N^\q$ such that  $\det_\ufd(\bK_\bullet^\nu(\Seq,M))=\det_\ufd(\bK_\bullet^{\nu_0}(\Seq,M))$ for every $\nu\geq\nu_0$. 
In other words, $\det_\ufd(\bK_\bullet^\nu(\Seq,M)$ stabilizes at a well-defined nonzero element of $\ufd/\ufd^\times\subseteq \field/\ufd^\times$, for $\nu$ large enough.

\begin{definition}\label{def:koszul:res}
Let $\ufd$ be a Noetherian UFD ring, $M$ a componentwise free $\ufd[\x]$-module with $\rdim_\ufd(M)=r$ and $\Seq=(\seq_0,\ld,\seq_r)$ a filter-regular sequence for $M$.
Then we define the \emph{M-resultant} $\Kres_\ufd(\Seq,M)\in\ufd/\ufd^\times$ of $\Seq$ with respect to $M$ by
\[
		\Kres_\ufd(\Seq,M)
		:=
		\det_\ufd(\bK_\bullet^\nu(\Seq,M))
\]
for $\nu\in\N^\q$ large enough.
\end{definition}

\begin{remark}\label{rmk:cohomology}
	The usual way of proving the stabilization of $\det_\ufd(\bK_\bullet^\nu(\Seq,M))$ is via the vanishing of certain cohomology
	modules \cite{GKZ94,Jouanolou:invariant}. 
	In a sense, our approach of relating it to $\chi_\ufd((M/(\Seq)M)_\nu)$ and interpreting it as a collection of local  Hilbert functions is more direct. 
	However, it should be noted that the stabilization of Hilbert functions to Hilbert polynomials is related to cohmological results such as the vanishing theorem of Serre \cite{Hartshorne1977}. 
	
	In the case of the Macaulay resultant, or more generally when $M$ is a polynomial algebra, we may take $\nu\geq\nu_0$ in \cref{def:koszul:res} for some explicit $\nu_0$ \cite[Theorem 2.2]{SchejaStorch}. 
	In general the value of $\nu_0$ depends on the Castelnuovo-Mumford regularity of $M$ \cite{CMregular:Castelnuovo,CMregular:Mumford,CMregular:Chardin}, see also \cite{CMregular:MaclaganSmith,CMregular:BotbolChardin} for multigraded Castelnuovo-Mumford regularity. 
\end{remark}

\begin{remark}\label{rmk:recovery}
The theory of \cite{Chardin1993} is recovered with the module $M=\ufd[\x]$, while the theory of \cite{RemondCh5} corresponds to the elimination ring $\ufd=\kk[\u]$, and the module $M=(\kk[\x]/I)\otimes\kk[\u]$, where  $\kk$ is a field and $I\subseteq\kk[\x]$ is a multihomogeneous ideal of $\kk[\x]$ (see \cref{sec:res:remond}).
\end{remark}

\begin{remark}\label{rmk:hypotheses}
For the sake of simplicity in this paper we usually assume that $M$ is a componentwise free $\ufd[\x]$-module and that $\ufd$ is a Noetherian UFD ring.
This is enough for our purposes, because of \cref{rmk:recovery}. 	
However our constructions, conveniently adapted, can be performed under weaker hypotheses, for example if $M$ is just projective over $\ufd$ and $\ufd$ is any integrally closed Noetherian integral domain. 
See for example \cite[Appendix A]{GKZ94} for a general definition of the Cayley determinant. 
Of course our presentation extends to the case in which $M$ is a multigraded module over some multigraded ring that is \emph{standard graded} (terminology of \cite{Trung2010}), i.e. that is generated over $R_\zero$ by elements with minimal multidegree. 
We have not attempted to cover the case of polynomial algebras $\ufd[\x]$ whose variables have arbitrary weight/multidegree. 
For the reader interested in this case, we refer to \cite{SchejaStorch,SchejaStorch:2001book}.
\end{remark}

\begin{remark}\label{rmk:schemes}
	As we mentioned in the introduction, the theory of resultants formulated for modules allude to a generalization to vector bundles over schemes. 
	This point of view is adopted for the mixed resultants in \cite[Chapter 3, Sec. 3]{GKZ94}, but some comments  are in order.
	Indeed, while the classical resultants and the mixed resultants are always irreducible, in the theory of R\'emond and of this paper, they might not be \cite[Example 1.31]{Nullstellen}.
	The reason is that in multiprojective setting the relevant line bundles come from projection on factors and thus they are not very ample.
	This forces one to allow multiplicities, in oder to have a well-behaved theory, including, for example, an analogue of \cite[Theorem 3.10]{GKZ94}.
\end{remark}


\subsection{R\'emond's definition of the resultant}
\label{sec:res:remond}

Let $\kk$ be a field, $\kk[\x]$ a multigraded polynomial ring as in \cref{sec:multi:poly}, $I$ a multihomogeneous ideal of $\kk[\x]$ and $M=\kk[\x]/I$.
For every multidegree $\d\in\N^\q$ we denote by $\mfk{M}_{\d}$ the collection of  monomials of multidegree $\d$ in the variables $\x$. 
Let $r=\rdim_\kk(M)$ and let $\dd=(\d^{(0)},\dots,\d^{(r)})$ be a collection of nonzero multidegrees. 
For $i=0,\ldots,r$ and $\mfk{m}\in\mfk{M}_{\di}$ we introduce a variable $u^{(i)}_\m$. 
The collection of variables $\u=(u^{(i)}_\m:\ 0\leq i\leq r,\ \mfk{m}\in\mfk{M}_{\di} )$ is called the collection of \emph{generic coefficients}. 
For $i=0,\ldots,r$ we also consider the subcollection $\u^{(i)}=(u^{(i)}_\m:\ \mfk{m}\in\mfk{M}_{\di} )$ and the \emph{generic polynomial} of multidegree $\di$ defined by
\[ 
   U_i:=\sum_{\m\in \mfk{M}_{\di}} u^{(i)}_\m \m,
\]
which is a multihomogeneous element of multidegree $\di$ in the polynomial ring $\kk[\u^{(i)}][\x]$.

Let $M[\u]:=M\otimes_\kk\kk[\u]$, $\uoU=(U_0,\ld,U_r)$ and $\mcl{M}(I):=M[\u]/(\uoU)M[\u]$. We observe that $\kk[\u]$ is an UFD ring and that $M[\u]$ is a componentwise free $\kk[\u][\x]$-module such that $\rdim_{\kk[\u]}(M[\u])=\rdim_{\kk}(M)=r$. 
In \cite{RemondCh5} it is proved, using multihomogeneous elimination, that $\mcl{M}(I)$ is a torsion $\kk[\u]$-module and that $\chi_{\kk[\u]}(\mcl{M}(I)_\nu)$ is equal, for $\nu\in\N^\q$ large enough, to a fixed element $\Rres_\dd(I)\in\kk[\u]$, called the \emph{resultant form} of index $\dd$ attached to $I$.
The aim of this paragraph is to prove the following.

\begin{theorem}\label{res:koszul=remond}
With the notation above, $\Rres_\dd(I)= \Kres_{\kk[\u]}(\uoU,M[\u])\ (\bmod\, \kk^\times)$.
\end{theorem}

To prove this, we adapt to our situation a classical result about the generic polynomials \cite[pp. 6-8]{Jouanolou1980}, saying that $\uoU$ is a filter-regular sequence for the $\N^\q$-graded $\kk[\u][\x]$-module $M[\u]$. By means of \cref{ann=det} this will imply that $\Rres_\dd(I)$ coincides with the M-resultant (see \cref{def:koszul:res}), up to elements of $\kk[\u]^\times=\kk^\times$, and so \cref{res:koszul=remond}.

\begin{lemma}\label{lemma:generic}
Let $R$ be any commutative ring, $M$ an $R$-module and $S$ a finite set. Let $(r_i)_{i\in S}$ be a set of elements $r_i\in R$ and $\v=(v_i)_{i\in S}$ a collection of independent variables, both indexed by $S$.
Denote $M[\v]:=M\otimes_R R[\v]$, let $J$ be the ideal of $R$ generated by $(r_i)_{i\in S}$ and let $V:=\sum_{i\in I} r_iv_i\in R[\v]$.
Then $(0:_{M[\v]}V)\subseteq (0:_M J^\infty)[\v]$, where $(0:_M J^\infty):=\bigcup_{n\in\N}(0:_M J^n)$.
\end{lemma}
\begin{proof}
We consider the elements of $M[\v]$ as polynomials in the variables $\v$ and with coefficients in $M$. 
More precisely, we consider the $\N^{|S|}$-grading on $R[\v]$ (and thus on $M[\v]$) induced by requiring that all elements of $R$ have degree $\zero$ and that for all $i\in S$ the element $v_i$ has degree $\mbf{e}_{i}$, where the $\mbf{e}_{i}$ are the canonical basis elements of $\N^{|S|}$ and $\zero$ is the trivial element.
Let now $m\in (0:_{M[\v]}V)$, so that $mV=0$ in $M[\v]$. We write $m=\sum_{\alpha\in\N^{|S|}} m_\alpha \v^\alpha$ and we will eventually prove that $m_\alpha\in (0:_M J^\infty)$ for every $\alpha$.

Fix $i\in S$.
Let $\LEX_i$ be a monomial lexicographic order on $\N^{|S|}$ so that $\mbf{e}_i\great \mbf{e}_j$ $\forall j\in S$, $j\neq i$.
We now prove that $\forall \alpha\in\N^{\abs{S}}$ $\exists n\in \N$ such that $m_\alpha r_i^n\in M$.
By contradiction, let $\alpha$ be a counterexample to this claim, maximal with respect to $\LEX_i$.
Comparing terms of multidegree $\alpha+\mbf{e}_i$ in the equality $mV=0$, we see that
\begin{equation}\label{eq:lemma:generic}
    m_\alpha r_i +\sum_{j\neq i,\ \alpha_j\neq 0} m_{\alpha+\mbf{e}_i-\mbf{e}_j} r_j =0.
\end{equation}
We notice that all the $\alpha+\mbf{e}_i-\mbf{e}_j$ appearing in this formula, if any, are bigger than $\alpha$ with respect to $\LEX_i$. Therefore by assumption  the corresponding $m_{\alpha+\mbf{e}_i-\mbf{e}_j}$ vanish when multiplied by certain power of $r_i$. Thus, if we multiply both sides of the equation \eqref{eq:lemma:generic} by a suitable power of $r_i$ we get a contradiction.
Let now $N\in\N$ be  big enough, so that $\forall\alpha$ $\forall i\in S$ we have $m_\alpha r_i^N=0$. Then we can deduce that for every $\alpha$ we have $m_\alpha \in (0:_M J^{N\cdot |S|})$. 
\end{proof}

\begin{corollary}\label{filter-regular:generic}
With the notation above we have that $\uoU=(U_0,\ld,U_\r)$ is a filter-regular sequence for $M[\u]$ in $\kk[\u][\x]$.
\end{corollary}
\begin{proof}
For $i=0,\ld,\r$ let $\wt\u_i:=\u\setminus \u^{(i)}$, let $\wt M_i:=M[\wt\u_i]/(U_0,\ld,U_{i-1})M[\wt\u_i]$ and let $M_i:=M[\u]/(U_0,\ld,U_{i-1})M[\u]$, so that $M_i=\wt M_i[\u^{(i)}]$. 
\Cref{lemma:generic} with $\v=\u^{(i)}$, $J=(\mfk{M}_{\di})$ and $V=U_i$ 
gives that for every  $m\in(0:_{M_i} U_i)$ there is $N(m)\in\N$ such that $m\mfk M_{N(m)\di}=0$ in $M_i$, and so that $m\kk[\u][\x]_\nu=0$  for all $\nu\geq N(m)\di$. 
Since $\kk[\u][\x]$ is Noetherian,  $M_i$ is Noetherian as well, and so $(0:_{M_i} U_i)$ is generated over $\kk[\u][\x]$ by finitely many multihomogeneous elements $m_1,\ld,m_\ell$, respectively with multidegrees $\nu_1,\ld,\nu_\ell$.  
Then, $(0:_{M_i} U_i)_\nu=\sum_{j=1}^\ell m_j\kk[\u][\x]_{\nu-\nu_j} = 0$ for every $\nu\in \N^\q$ such that $\nu\geq\nu_j+N(m_j)\di$, $\forall j=0,\ld,\ell$.
This means that $U_i$ is filter-regular for the module $M_i$.
\end{proof}

\begin{remark}\label{rmk:irreducible}
	Despite the lost of irreducibility, it is comforting to acknowledge that the theory of R\'emond resultants retains some of the essential features of the theory of resultants, such as the computability via Cayley determinants.
	As we have seen, this is because the Cayley determinant, thanks to \cite[Theorem 30, Appendix A, p.493]{GKZ94}, detects the multiplicities in the divisor of a complex.
\end{remark}


\section{Lower bounds for the multipicity of the resultant}
\label{sec:main}

%

\subsection{The order function induced by a prime ideal}
\label{sec:main:order}

Let $\ufd$ be a Noetherian integral domain, let $\p$ be a nonzero prime ideal of $\ufd$, and let $\m_\p:=\p\ufd_\p$ be the maximal ideal of the localization $\ufd_\p$ of $\ufd$ at $\p$. 
For every $n\in\N$ the \emph{$n$-th symbolic power} of $\p$ is $\p^{(n)}:=\m_\p^n\cap\ufd$.
The following proposition (see \cite[Vol.1, Ch. IV, Sec. 12]{Zariski} gives alternative definitions for symbolic powers.

\begin{proposition}\label{symbolic:power} 
We have $\p^{(n)}=\{a\in\ufd:\ \exists b\in\ufd - \p \text{ with } ab\in\p^n\}$.
Moreover, $\p^{(n)}$ is the smallest $\p$-primary ideal of $\ufd$ that contains $\p^n$. 
In particular if $\p$ is maximal then $\p^n=\p^{(n)}$.
\end{proposition}  
As a consequence of Krull's intersection theorem we have $\bigcap_{n=0}^\infty \m_\p^n=\{0\}$, and so we can consider the \emph{order function}  $\ordp:\ufd_\p\to\N\cup\{+\infty\}$ associated to the filtration $\{\m_\p^n\}_{\n\in\N}$, given by $\ordp(0)=+\infty$ and $\ordp(a)=n$ if $a\in\m_\p^n-\m_\p^{n+1}$ \cite[Ch. III, Sec.  2.2]{bourbaki}.
The order function $\ordp$ satisfies $\ordp(a+b)\geq\min\{\ordp(a),\ordp(b)\}$ and $\ordp(ab)\geq\ordp(a)+\ordp(b)$ for all $a,b\in\ufd_\p$.
Moreover, it satisfies a \emph{weak homomorphism property}: if $a,b\in\ufd_\p$ and  $\ordp(b)=0$, then $\ordp(ab)=\ordp(a)$.
%
The restriction of $\ordp$ to $\ufd$ is the order function with respect to the filtration $\{\p^{(n)}\}_{n\in\N}$. 
Moreover, if $\wb{a}\in\ufd/\ufd^\times$ we define $\ordp(\wb{a})$ to be the order of any element of $\ufd$ representing $\wb{a}$. 
This is a good definition because $\ordp(u)=0$ for every $u\in\ufd^\times$.

\begin{remark}\label{rmk:order:meaning}
Geometrically speaking, an element $a\in\ufd_\p$ is a rational function over $\op{Spec} \ufd$, regular  in a neighbourhood of $\p$.
Then $\ordp(a)$ is interpreted as the multiplicity of vanishing of $a$ at $\p$.  
See also the Zariski-Nagata Theorem \cite[Chapter 3.9]{Eisenbud} about this interpretation.
\end{remark}


%

\subsection{The multiplicity of the resultant along a prime ideal}
\label{sec:main:chardin}

Let $\ufd$ be a Noetherian UFD ring with fraction field $\field$, let $M$ be a componentwise free $\ufd[\x]$-module as in \cref{sec:multi:poly}, with $\rdim_\ufd(M)=r$ and let $\Seq=(\seq_0,\ld,\seq_r)$ be a filter-regular sequence in $\ufd[\x]$ for $M$.
Let also $\p$ be a prime ideal of $\ufd$, $\vk=\ufd_\p/\p\ufd_\p$ the residue field, $\vMod=M\otimes_\ufd \vk$,  $\pi:\ufd[\x]\to\vk[\x]$ the natural projection, $(\pi(\Seq))$ the ideal of $\vk[\x]$ generated by $\pi(\seq_0),\ld,\pi(\seq_r)$ and $N:=\vMod/(\pi(\Seq))\vMod$.

\begin{theorem}\label{thm:main:chardin}
With the above notation, consider the resultant $\Kres_\ufd(\Seq,M)\in\ufd/\ufd^\times$ as in \cref{def:koszul:res} and suppose that $\fgr((\pi(\Seq)),\vMod)=r$. 
Then $$\ordp(\Kres_\ufd(\Seq,M))\geq \rdeg_\vk(N).$$
\end{theorem}

\begin{remark}\label{rmk:chardin:enhancement}
	If $M=\Z[\x]$ and $p$ is a prime number, $\ell:=\rdeg_{\mbb F_p}(N)$ counts with multiplicity the number of common zeros modulo $p$ of the polynmials $\Seq$. 
	Then from \cref{thm:main:chardin} we recover the result of Chardin mentioned in the introduction: the resultant $R(\Seq):=\Kres_\ufd(\Seq,M)$ is an integer divisible by $p^\ell$. 
	Now we observe that we can say more if instead we use the module $M=\Z[\u,\x]$, the sequence of generic polynomials $\uoU$ and the U-resultant $R(\uoU)\in \Z[\u]/\{\pm 1\}$. 
	Let $\p\subseteq \Z[\u]$ be the kernel of the morphism that maps the generic coefficients $u^{(i)}_\m$ to the corresponding coefficients of $\seq_i$, composed with the reduction modulo $p$. 
	Then again $N\cong \mbb F_p[\x]/(\Seq)$ has relevant degree equal to $\ell$ and so we get 
	\begin{equation}\label{eq:cha:enh}
		\ordp(R(\uoU))\geq \ell.
	\end{equation}
	If we expand in Taylor series the polynomial $R(\uoU)$, at the point corresponding to the coefficients of $\Seq$, we rediscover that $p^\ell\divides R(\Seq)$, but we also prove more: all partial derivatives $\de_{u^{(i)}_\m} R(\Seq)$ are divisible by $p^{\ell-1}$ and more generally all iterated derivatives $\de_{\u}^\alpha R(\Seq)$ of order $\abs{\alpha}<\ell$ are divisible by $p^{\ell-\abs\alpha}$.
\end{remark} 

\begin{remark}
We recall from \cref{sec:multi:regular} that $\fgr((\pi(\Seq)),\vMod)$ is the maximal length of a filter-regular sequence for $\vMod$ made of elements of $(\pi(\Seq))$, and we recall from \cref{sec:multi:hil} that the relevant dimension $\rdim_\ufd(M)$ is the total degree of the Hilbert polynomial $H_M$. 
By \cref{filter-regular:hilbert:rdim} $\fgr((\pi(\Seq)),\vMod)$ can be seen as a \emph{codimension} of $N$ with respect to $\vMod$. 
Since $\rdim_\ufd(M)=r$ we have $\rdim_\vk(\vMod) = r$ and this, together with $\fgr((\pi(\Seq)),\vMod)=r$, implies  $\rdim_\vk(N)\leq 0$. Therefore $H_N$ is constant and $\rdeg_\vk(N)=H_N$ is defined.
\end{remark}

The proof of \cref{thm:main:chardin} relies on the computation of the resultant via Cayley determinants and Koszul complexes, as done in \cref{sec:res:det}, on an adaptation of techniques already used by Chardin in \cite{Chardin1993}, and on the following easy lemma.

\begin{lemma}\label{lemma:order:det}
Let $D$ be an $s\times s$ square matrix with entries in $\ufd_\p$, 
let $\wb{D}$ be the matrix with entries in $\vk$ obtained from $D$ by reduction modulo $\vm$
and let $\op{corank}(\wb{D})$ denote the codimension of the image of the $\vk$-linear map represented by $\wb{D}$.
Then $\ordp(\det(D))\geq \op{corank}(\wb{D})$.
\end{lemma}
\begin{proof}
Since $\vk$ is a field, we can find two
 invertible $s\times s$ matrices $A$, $B$ with coefficients in $\vk$ such that $A\wb{D}B$ is a block matrix $\left(\begin{smallmatrix} I & 0 \\ 0 & 0\end{smallmatrix}\right)$ with the first block being square of size $s-\op{corank} (\wb{D})$. We lift arbitrarily $A$ and $B$ to matrices $\wt{A}$ and $\wt{B}$ with entries in $\ufd_\p$ and we notice that $\ordp(\det(\wt{A}))=\ordp(\det(\wt{B}))=0$.  
Then all the entries of the last $\op{corank} (\wb{D})$ columns (or rows) of the matrix $\wt{A}D\wt{B}$ belong to $\vm$, and thus we obtain from Laplace's expansion that $\ordp(\det(\wt{A}D\wt{B}))\geq\op{corank}(\wb{D})$. 
We conclude what we wanted using the multiplicativity of the determinant and the weak homomorphism property of the order function $\ordp$ (see \cref{sec:main:order}).
\end{proof}

\begin{proof}[Proof of \cref{thm:main:chardin}]
If $\dim_\vk(N)=-1$, then $\rdeg_\vk(N)=0$ and the thesis is trivial. 
Therefore, we suppose $N$ is not eventually zero. 
Since all the homogeneous components of $M$ are free $\ufd$-modules of finite rank, 
for every $\nu\in\N^\q$ the complex $\bK^\nu_\bullet:=\bK^\nu_\bullet(\Seq,M)$ is a finite complex of free $\ufd$-modules of finite rank.
We can therefore choose a system $\{\ub^{(\nu)}_{\, p}\}_{0\leq p\leq r+1}$ of $\ufd$-bases for the modules $\bK_p^\nu(\Seq,M)$.
When we change scalars from $\ufd$ to $\vk$ we can consider the induced $\vk$-bases, which we still call $\ub^{(\nu)}_{\, p}$, for the $\vk$-vector spaces $\wb{\bK}_p^\nu:=\bK_p^\nu(\pi(\Seq),\vMod)\cong\bK_p^\nu\otimes_\ufd\vk$.
Since $\Seq$ is filter-regular for $M$, we have by \cref{ann=det} that for $\nu$ large enough the complex $\bK^\nu_\bullet$ is generically exact and $\det_\ufd(\bK^\nu_\bullet)=\Kres_\ufd(\Seq,M)\ (\bmod\, \ufd^\times)$.
In addition to this, $\oH_0(\wb{\bK}^\nu_\bullet)=N_\nu$ for every $\nu\in\N^\q$ and so $\dim_\vk(\oH_0(\wb{\bK}^\nu_\bullet))=\rdeg_\vk(N)$ for $\nu$ large enough.
Moreover, since $N$ is not eventually zero and $\fgr((\pi(\Seq)),\vMod)=r$, \cref{northcott} (ii) implies that the homology modules $\oH_p(\wb{\bK}^\nu_\bullet)$ vanish for $p\geq (r+1)-r+1=2$ and $\nu$ large enough.
Let $\nu\in\N^\q$ such that all the above requirements hold for $\nu'\geq\nu$ and denote by $\wb{\de}^\nu_p$ the differentials of $\wb{\bK}^\nu_\bullet$, induced by the differentials $\de^\nu_p$ of $\bK^\nu_\bullet$.
By the vanishing of the higher homology, we can find by elementary linear algebra (see for example \cite{Chardin1993}) a partition of the bases $\ub^{(\nu)}_{\, p}=\ub^{(\nu)}_{\, p,1}\cup\ub^{(\nu)}_{\, p,2}$ for $p=1,\ld,r+1$, with $\ub^{(\nu)}_{\, r+1,1}=\emptyset$, inducing decompositions of $\wb{\bK}^\nu_p$ and $\bK^\nu_p$, such that for $p=2,\ld,r+1$ the matrix representations of the differentials $\wb{\de}^\nu_p$ (resp. $\de^\nu_p$) take the form $\left(\begin{smallmatrix} \wb{a}_p & \wb{\phi}_p\\ \wb{b}_p & \wb{c}_p \end{smallmatrix}\right)$ (resp. $\left(\begin{smallmatrix} a_p & \phi_p\\ b_p & c_p \end{smallmatrix}\right)$), where $\wb{\phi}_p$ (resp. $\phi_p$) is a square matrix with entries in $\vk$ (resp. $\ufd$) and nonzero determinant (resp. determinant in $\ufd-\p$) for $p=2,\ld,r+1$.
For $p=0$ we consider the trivial partition $\ub^{(\nu)}_{\, 0}=\ub^{(\nu)}_{\, 0}\cup\emptyset$, that induces a block matrix representation of $\de^\nu_1$ of the form $\left(\begin{smallmatrix} a_1 & \phi_1 \end{smallmatrix}\right)$.
From the fact that the complex $\bK^\nu_\bullet$ is generically exact we deduce that also the matrix $\phi_1$ must be square.
Then, by definition, the Cayley determinant of $\bK^\nu_\bullet$ with respect to the above choices of bases and partitions is given by
\[
    \det_\ufd(\bK^\nu_\bullet) = \prod_{i=1}^{r+1}\det(\phi_i)^{(-1)^{i+1}} .
\]
By the above construction we have $\ordp(\phi_i)=0$ for $i=2,\ld,r+1$ and from \cref{lemma:order:det} we have $\ordp(\phi_1)\geq \dim_\vk(\oH_0(\wb{\bK}^\nu_\bullet))$.
By the above choice of  $\nu$ and the weak homomorphism property of the order function $\ordp$ we deduce
\[
		\ordp(\det_\ufd(\bK^\nu_\bullet))\geq \rdeg_\vk(N).
\]
Since $\Kres_\ufd(\Seq,M)=\det_\ufd(\bK^\nu_\bullet)\  (\bmod\, \ufd^\times)$ for $\nu$ large enough, we conclude what we wanted.
\end{proof}


\subsection{The order of vanishing at a sequence of polynomials}
\label{sec:main:roy}

In this paragraph we focus specifically on the R\'emond resultant attached to a multihomogeneous ideal as in \cref{sec:res:remond} and therefore we work in a multiprojective setting as in \cref{sec:multi:geometry}.
Let $\kk[\x]$ be a multigraded polynomial ring with $\kk$ an infinite field, let $I\subseteq\kk[\x]$ be a multihomogeneous ideal with $\dim\mcl{Z}(I)=\r$.
Let $\dd=(\dz,\ld,\dr)$ be a collection of nonzero multidegrees, and let $\kk[\u]$ and $\Rres_\dd(I)\in\kk[\u]$ be as in \cref{sec:res:remond}.

For every $(\r+1)$-tuple of polynomials $\Seq=(\seq_0,\ldots,\seq_\r)$ of $\kk[\x]$ with multidegrees prescribed by $\dd$ there exists, by the universal property of polynomial rings, a unique $\kk$-algebra map $\eval_{\Seq}: \kk[\u]\to \kk$ that maps the generic coefficients $u^{(i)}_\m$ to the corresponding coefficients of $\seq_i$ and restricts to the identity on $\kk$. This also means that every $R\in \kk[\u]$  induces a map 
\[
  R(\cdot):\ \kk[\x]_\dz\times \dots \times \kk[\x]_\dr \longrightarrow \kk
\]
given by $R(\Seq):=\eval_{\Seq}(R)$. 
We observe that the kernel of the map $\eval_{\Seq}$ is a maximal ideal. Then let $\ordF:\kk[\u]\to \N\cup\{+\infty\}$ be the order function corresponding to it as in \cref{sec:main:order}.
\begin{remark}\label{rmk:ord}
One can show that for every $R\in\kk[\u]$ the value $\ordF (R)$  is the largest power of $t$ dividing $T=R(\Seq+t\uoU)\in\kk[\u][t]$, where $\uoU=(U_0,\ld,U_\r)$ is the sequence of generic polynomials as in \cref{sec:res:remond}, and $T$ is defined as above by means of the universal property of polynomial rings.
\end{remark}

\begin{theorem}\label{thm:main:roy:CM}
Let $J$ be a multihomogeneous ideal of $\kk[\x]$ such that  $I\subseteq J$ and $\dim\mcl{Z}(J)=0$.
Suppose also that, for every $i=0,\ld,\r-1$, we have $\dim\mcl{Z}(J_\di)=0$
and that, for every relevant $\mfk{p}\in\Ass_{\kk[\x]}(\kk[\x]/J_{\di}\kk[\x])$, the local ring (module) $(\kk[\x]/I)_\mfk{p}$ is Cohen-Macaulay of (Krull) dimension $\r$.
Then the resultant form $\res_\dd(I)$ vanishes to order at least $\deg(J)$ at each $(\r+1)$-tuple $\Seq=(\seq_0,\ldots,\seq_\r)\in J_{\dz}\times\dots\times J_{\dr}$, i.e. $\ordF (\Rres_\dd(I))\geq \deg(J)$.
\end{theorem}

\begin{remark}
Geometrically speaking, we require that $\mcl{Z}(J_\di)$ is supported on a finite set of points, located on components of $\mcl{Z}(I)$ with maximal dimension, and that $\mcl{Z}(I)$ has mild singularities at these points.
\end{remark}

\begin{proof}[Proof of \cref{thm:main:roy:CM}] 
We adapt an idea from \cite[Theorem 5.2]{Roy2013} %
 and consider the affine space $\mbb{A}_\dd$ over $\op{Spec}\kk$ corresponding to the finite dimensional $\kk$-vector space $\kk[\x]_\dz\times\dots\times\kk[\x]_\dr$. 
Then $\mcl{V}=J_{\dz}\times\dots\times J_\dr$ is a $\kk$-vector subspace of $\mathbb{A}_\dd$ and so it is an algebraic subset of it, irreducible and closed in the Zariski topology.
We pospone the proof of the following fact.

\begin{lemma}\label{claim:CM:filter-regular}
Under the hypotheses of \cref{thm:main:roy:CM} there exists a Zariski dense subset $\mcl{U}$ of $\mcl{V}$ such that for every $\Seq=(\seq_0,\ld,\seq_{\r})\in\mcl{U}$ the subsequence $(\seq_0,\ld,\seq_{\r-1})$ is filter-regular for $M=\kk[\x]/I$.
\end{lemma}

Given this fact, we apply \cref{thm:main:chardin} to get $\ordF(\Rres_\dd(I))\geq \rdeg_\kk(\kk[\x]/(I,\Seq))$ for every $\Seq\in\mcl{U}$.
From $(I,\Seq)\subseteq J$ and $\deg(J):=\rdeg_\kk(\kk[\x]/J)$ we deduce in particular that $\ordF(\Rres_\dd(I)) \geq \deg(J)$ for every $\Seq\in\mcl{U}$.
To conclude it then suffices to see that the set $\{\Seq\in\mbb{A}_\dd:\, \ordF(\Rres_\dd(I)) \geq \deg(J)\}$ is Zariski closed.
This is true because this is the common zero locus of a collection of polynomial functions $\{\mcl{D}\,\Rres_\dd(I)\}_{\mcl{D}}\subseteq\kk[\u]$, where $\mcl{D}$ ranges through the differential operators on $\kk[\u]$ which are partial derivatives of order at most $\deg(J)$.
\end{proof}

\begin{proof}[Proof of \cref{claim:CM:filter-regular}]
Let $\cA=\bigcup_{i=0}^{\r-1} \Ass_{\kk[\x]}(\kk[\x]\slash J_{\di}\kk[\x])$. 
We will prove that for every $i=0,\ld,{\r}$ there exists a Zariski dense subset $\mcl{U}_i\subseteq\mcl{V}$ with the following properties:
\begin{enumerate}[(i)]
\item
for every $\Seq=(\seq_0,\ld,\seq_{\r})\in\mcl{U}_i$ the subsequence $\Seq_{(i)}:=(\seq_0,\ld,\seq_{i-1})$ is filter-regular for $M=\kk[\x]/I$;
\item
for every relevant prime $\mfk{p}\in\cA$ the module $M_{\Seq,i}:=M/(\Seq_{(i)})M$ is locally Cohen-Macaulay at $\mfk{p}$;
\item
for every relevant prime $\mfk{p}\in\cA$ and every $\mfk{q}\in\Ass_{\kk[\x]} (M_{\Seq,i})$ with $\mfk{q}\subseteq \mfk{p}$ we have $\dim\mcl{Z}(\mfk{q})=\r-i$.
\end{enumerate}
The degenerate case $i=0$ is provided by the hypothesis and $\mcl{U}_0=\mcl{V}$.
Let $\mcl{U}_{i}$ satisfy the requirements for some  $i\leq \r-1$, let $\Seq\in\mcl{U}_{i}$ and consider the following finite collection of  $\kk$-subspaces of $J_\di$:
\[
	\mathcal{S}_{\Seq} = \{ \mfk{q}\cap J_\di:\, \mfk q\in \Ass_{\kk[\x]}(M_{\Seq,i}),\,\dim\mcl{Z}(\mfk q)\geq 0\}.
\] 
Since $\dim\mcl{Z}(J_\di)=0$ we see that if $\mfk{q}$ is any multihomogeneous prime of $\kk[\x]$ containing $J_\di$, then either $\mfk{q}$ is irrelevant ($\dim\mcl{Z}(\mfk{q})=-1$) or $\dim\mcl{Z}(\mfk{q})=0$ and $\mfk{q}\in\cA$ because in particular $\mfk{q}$ is minimal over $J_\di$. 
In either case, also by condition (iii) above, no such $\mfk q$ appears in the definition of $\mcl{S}_{\Seq}$.
Therefore $\mcl{S}_{\Seq}$ is a finite collection of proper $\kk$-subspaces of $J_\di$.
Since $\kk$ is an infinite field, their union $S_\Seq:=\cup\mcl S_\Seq$ is a proper Zariski-closed subset of $J_\di$.
We now define
\[
\mcl{U}_{i+1} := \bigcup_{\Seq\in\mcl{U}_i} \{(\seq_0,\ld,\seq_{i-1})\}\times (J_\di-{S}_{\Seq})\times \{(\seq_{i+1},\ld,\seq_\r)\}. 
\]
For $\Seq\in\mcl U_i$ the closure of $J_\di-{S}_{\Seq}$ is all of $J_{\di}$, so in particular it contains $\seq_i$.
Then $\mcl U_{i+1}$ is dense in $\mcl{V}$, because its closure contains $\mcl{U}_i$.
For every $\Seq=(\seq_0,\ld,\seq_{\r})\in\mcl{U}_{i+1}$ the element $\seq_i$ is filter-regular for $M_{\Seq,i}$ by \cref{filter-regular:relevant} and for every relevant prime $\mfk{p}\in\cA$ it is a regular element for the localization $(M_{\Seq,i})_\p$, which is Cohen-Macaulay.
Therefore $(M_{\Seq,i+1})_\p$ is Cohen-Macaulay as well. Moreover, by unmixedness, all associated primes $\mfk{q}'$ of $M_{\Seq,i+1}$ containing $\mfk{p}$ are minimal ones. Since they are in particular minimal primes for the ideals $(\mfk{q},\seq_i)$, where $\mfk{q}$ is an associated prime of $M_{\Seq,i}$ containing $\mfk{p}$, every such $\mfk{q}'$ satisfies $\dim\mcl{Z}(\mfk{q}')=\r-i-1$. 
We can then continue by induction and we conclude what we wanted when $i=\r$.
\end{proof}

\begin{remark}\label{filter:regular:why}
In \cref{thm:main:chardin} we used the notion of $\fgr$, defined in terms of filter-regular sequences, instead of the more common notion of $\text{depth}$, involving regular sequences.
Indeed, the former is more natural (many of our statements are true `for $\d$ large enough') and more general (a regular sequence is also filter-regular).
Moreover, it was essential in order to prove \cref{claim:CM:filter-regular} (and so \cref{thm:main:roy:CM}), imposing only mild conditions on the multiprojective subvariety $\mcl{Z}(I)$. 
Namely, we assumed it to be locally Cohen-Macaulay (e.g. smooth is enough) at a finite number of points. 

In fact, to have the analogous statement with regular sequences, one needs $\mcl{Z}(I)$ to be arithmetically Cohen Macaulay (ACM), which means that the whole coordinate ring $\kk[\x]/I$ is Cohen-Macaulay (thus also at the irrelevant primes).
This is a strong global condition, but it is satisfied, for example, in the case $\mcl{Z}(I)=\PP^\n_\kk$ studied in \cite[Theorem 5.2]{Roy2013}.

We give an example, taken from \cite{thesisTuyl}, of a family of non-ACM varieties.
 Let $\q=1$, $n_1=2m-1$ and $I=(x_{2k}:\, 0\leq k\less m) \cap (x_{2k+1}:\, 0\leq k\less m)$.
Then $\mcl{Z}(I)$ corresponds to an $(m-1)$-dimensional projective variety but the $\kk[\x]$-module $\kk[\x]/I$ has only $\text{depth}=1$. Indeed, it's not possible to extend the regular sequence $\{x_0+x_1\}$, since after factoring it out, $x_0$ annihilates all the monomials.
For an example of a non-CM integral domain see \cite{mathoverflowCM}.%
\end{remark}

%
%
%
%

\section{Polynomials vanishing at prescribed directions}
\label{sec:application}


\subsection{Preliminaries on commutative algebraic groups}
\label{sec:application:group}

Let $G_1,\ldots,G_\q$ be connected commutative algebraic groups defined over $\C$. We recall that they are smooth quasi-projective varieties by the structure theorem of Chevalley and Barsotti
and their set of complex points $G_1(\C),\dots, G_\q(\C)$ have a structure of complex Lie groups.
Let $\wb{G}_1,\ldots,\wb{G}_\q$ be suitable projective compactifications of them, embedded in projective spaces by $\theta_i: \wb{G}_i\hookrightarrow \PP^{n_i}_{\C}$ for $i=1,\ldots,\q$. We then put $G=G_1\times\ldots\times G_\q$, $\wb{G} =\wb{G}_1\times\dots\times\wb{G}_\q$, $\PP^\n_{\C}=\PP^{n_1}_{\C}\times\dots\times \PP^{n_\q}_{\C}$ and $\theta=\theta_1\times\dots\times\theta_\q:\wb{G}\hookrightarrow\PP^\n_{\C}$. Thus, we consider $G$ as a Zariski open subscheme of a multiprojective reduced closed subscheme $\wb{G}$ of the multiprojective space $\PP^\n_{\C}$. For $i=1,\ldots, \q$ we consider in $\PP^{n_i}_{\C}$ a set of projective coordinates $\x_i=(x_{i,0},\ldots,x_{i,n_i})$ and the affine coordinate chart
$U_i$ defined by $\{x_{i,0}\neq 0\}$. We consider in $\PP^\n_{\C}$ the set of multiprojective coordinates $\x=(\x_1,\dots,\x_\q)$, the affine chart $U=U_1\times\dots\times U_\q$, and the multigraded coordinate ring $\C[\x]$. We denote by $\mfk{G}\subseteq \C[\x]$ the multihomogeneous ideal of definition of $\wb{G}$, which is a prime ideal because $\wb{G}$ is irreducible, being the closure of a connected algebraic group. We also let  $\pi_i:\PP^\n_{\C}\to\PP^{n_i}_{\C}$ and use the same symbol to indicate the projections $\wb{G}\to\wb{G}_i$ and $G\to G_i$. 

Let $T_eG(\C)=T_{e_1}G_1(\C)\times\dots\times T_{e_\q}G_\q(\C)$ be the tangent space at the identity, identified with the Lie algebra $\mfk{g}=\mfk{g}_1\times\dots\times\mfk{g}_\q$ of invariant derivations on $G(\C)$. This Lie algebra is commutative since the Lie group $G(\C)$ is commutative. Let $\Delta=\{\de_1,\ldots,\de_\nW\}\subseteq \mfk{g}$ be a set of linearly independent invariant derivations 
and let $\Sigma=\{\gamma_1,\ldots,\gamma_\np\}\subseteq G(\C)$ be a finite set of complex points of $G$. 
We assume that $\Sigma\subset U(\C)$ (see \cite[p.~492]{Masser1981} for how to reduce the general case to this one). 
For every $\sigma\in\N^\nW$ we define the differential operator $\de^\sigma=\de_1^{\sigma_1}\dots\de_\nW^{\sigma_\nW}$ of order $\abs{\sigma}=\sigma_1+\dots+\sigma_\nW$ and for every $\um\in\Z^\ell$ we define the point $\um\gamma=m_1\gamma_1+\dots+m_\np\gamma_\np$.


Since we assumed that $\Sigma$ is contained in the affine chart $U=\{x_{1,0}\neq 0\}\cap\ldots\cap \{x_{\q,0}\neq 0\}$ we can give the following definition, as is done in \cite{Fischler2005} for the homogeneous case.

\begin{definition}\label{def:ev}
Given $\Sigma, \Delta$ as above and a positive integer $T$, we define for every multidegree $\d$ the evaluation operator
\[
\begin{array}{rccc}
\op{ev}_{\Sigma,T,\d}: & \C[\mbf{x}]_\d& \too &\C^{|\Sigma|\TW}\\
   &    P   &\mapsto    &\left(\de^\sigma\left(\frac{P}{x_{1,0}^{d_0}\dots x_{\q,0}^{d_\q}}\right)(\gamma):\ \abs{\sigma}\less T,\gamma\in \Sigma \right)
\end{array}
\]
\end{definition}

\begin{remark}
One can slightly generalize the datum of $\Delta,\Sigma,T$ introducing the concept of a \emph{ponderated set}, as in \cite{Philippon1996} or \cite{Galateau2014}. Moreover one can enlarge this setting to quasi-projective varieties with an action of $G$ \cite{Nakamaye1995} or even to non-commutative algebraic groups \cite{Huicochea2015}, under suitable hypothesis on the projective embedding.
\end{remark}


\subsection{The interpolation ideal} 
\label{sec:application:ideal}

Throughout this paragraph we keep the setting and the notations for $G,\theta,\Sigma,\Delta,T$ introduced in \cref{sec:application:group}. We define in this multiprojective setting the main ideal $\IST$ of the theory of interpolation on commutative algebraic groups, which is the ideal generated by the multihomogeneous polynomials vanishing in $\Sigma$ with order $T$ in the directions prescribed by $\Delta$. We then describe the multiprojective subscheme it defines and its relation with the surjectivity of the map $\op{ev}_{\Sigma,T,\d}$ introduced in \cref{def:ev}.

\begin{definition}\label{def:IST}
For every multidegree $\d\in\N^\q$ we let
\[
\IST_\d:=\op{ker}(\op{ev}_{\Sigma,T,\d})
\]
and then we define  $\IST:=\bigoplus_{\d\in\N^\q}\IST_\d$.
\end{definition}

We observe that $\IST$ is a multihomogeneous ideal of $\C[\x]$ which contains $\mfk{G}$.
The following result is a `trivial' form of an interpolation lemma. In general the objective of an interpolation lemma is to achieve better estimates for the multidegree $\d_{ev}$. Here we essentially reproduce Lemma 4.2 of \cite{Fischler2005} in multihomogeneous setting.

\begin{proposition}\label{lemma:interpolation:trivial}
Let $\d_{ev}\in\N^\q$ have all its coordinates equal 
to $T\abs{\Sigma}$. 
Then for every $\d\geq\d_{ev}$ the map $\ev_{\Sigma,T,\d}$ is surjective.
\end{proposition}

\begin{proof}
Let $\uno:=(1,\ldots,1)\in\N^\q$ as in \cref{sec:multi:def}, and $\d\geq\d_{ev}= T\abs{\Sigma} \uno$. 
For every $\nu\in\N^\q$ let $z(\nu)\in\kk[\x]_\nu$ be given by the formula
\[
		z(\nu) := \prod_{p=1}^\q  x_{p,0}^{\nu_p} .
\]
Let $\gamma,\delta\in\Sigma$ be distinct and let $i\in\{1,\ld,\nW\}$.
We now exhibit the existence of polynomials $L_{\gamma,\delta}, M_{\gamma,i}\in\kk[\x]_\uno$ such that: 
\begin{enumerate}[(i)]
\item
$L_{\gamma,\delta}$ vanishes at $\delta$ and not at $\gamma$; 
\item
$M_{\gamma,i}$ vanishes at $\gamma$ and $\de_j(M_{\gamma,i}/z(\uno))(\gamma)=\delta_{i,j}$ for all $j=\{1,\ld,\nW\}$, 
\end{enumerate}
where $\delta_{i,j}$ is Kronecker's symbol.
Then, given $\gamma\in\Sigma$ and  $\sigma\in\N^\nW$ with $\abs{\sigma}\less T$, we construct a polynomial $P_{\gamma,\sigma}\in\kk[\x]_\d$ such that:
\begin{enumerate}[(i)]
\item
$\de^\sigma (P_{\gamma,\sigma}/z(\d))(\gamma)\neq 0$;
\item
$\de^\tau(P_{\gamma,\sigma}/z(\d))(\gamma)= 0$ for every $\tau\leq\sigma$ with $\tau\neq \sigma$;
\item
$\de^\tau(P_{\gamma,\sigma}/z(\d))(\delta)= 0$ for every $\delta\in\Sigma-\{\gamma\}$  and every $\tau\in\N^\nW$ with $\abs{\tau}\less T$.
\end{enumerate}
It is clear that these polynomials will witness the surjectivity of $\ev_{\Sigma,T,\d}$.

Since $\gamma\neq \delta$ there are $i,j,p$ with $1\leq p\leq \q$ and $0\leq i\less j\leq n_p$ such that the linear form $\delta_{p,i}x_{p,j}-\delta_{p,j}x_{p,i}$ vanishes at $\delta$ and not at $\gamma$. We thus define 
\[
L_{\gamma,\delta} := (\delta_{p,i}x_{p,j}-\delta_{p,j}x_{p,i})\prod_{p'\neq p} x_{p',0} . 
\]
Since the derivations $\de_1,\ld,\de_\nW$ are linearly independent, the following matrix, with $\nW$ rows and $\abs{\n}=n_1+\ld+n_\q$ columns,
\[
\left[\de_j\left(\frac{x_{p,k}}{x_{p,0}}\right)(\gamma)\right]_%
{\substack{j:\, 1\leq j\leq \nW\\ (p,k):1\leq p\leq \q,1\leq k\leq n_p}}
= 
\left[\de_j\left(\frac{x_{p,k}\prod_{p'\neq p} x_{p',0}}{z(\uno)}\right)(\gamma)\right]_%
{\substack{j:\, 1\leq j\leq \nW\\ (p,k):1\leq p\leq \q,1\leq k\leq n_p}}
\]
has rank $\nW$. Therefore for every $i\in\{1,\ld,\nW\}$ there is  $\wt{M}_{\gamma,i}\in \kk[\x]_\uno$ such that $\de_j(\wt{M}_{\gamma,i}/z(\uno))(\gamma)=\delta_{i,j}$ for all $j\in\{1,\ldots,\nW\}$.
Then we define $M_{\gamma,i}$ by adding to $\wt{M}_{\gamma,i}$ a suitable multiple of $z(\uno)$ so that $M_{\gamma,i}(\gamma)=0$.
Finally, we define, for $\gamma\in\Sigma$ and $\sigma\in\N^\nd$ with $\abs{\sigma}\less T$:
\[
	P_{\gamma,\sigma} 
	= 
			z(\d-(\abs{\sigma}+(\abs{\Sigma}-1) T)\uno)
	\prod_{i=1}^\nW 
			M_{\gamma,i}^{\sigma_i}
	\prod_{\delta\in\Sigma\bsl\{\gamma\}} 
			L_{\gamma,\delta}^{ T}.
\]
\end{proof}

The following proposition employs a qualitative modification of a long division algorithm from \cite{Roy2013}. 
\begin{proposition}\label{IST:generated}
Let $\d\in\N^\q$ such that $\ev_{\Sigma,T,\d}$ is surjective and let $\d'\in\N^\q$ such that 
 $\d'\geq\d+\uno$. Then $(\IST_{\d'})=\IST\cap\C[\x]_{\geq \d'}$.
\end{proposition}
\begin{proof}
Let $\d''\geq \d'$. We need to show that  $\IST_{\d''}=\C[\x]_{\d''-\d'}\IST_{\d'}$.
We denote by $(\mbf{e}_p)_{1\leq p\leq \q}$ the canonical basis of $\N^\q$ as in \cref{sec:multi:def}. We will prove the assertion assuming $\d''=\d'+\mbf{e}_p$ for some $p$. The general case then follows by induction because $\C[\x]_\mbf{a}\C[\x]_\mbf{b}=\C[\x]_{\mbf{a}+\mbf{b}}$ for every $\mbf{a},\mbf{b}\in\N^\q$. Let $Q$ be any element of $\IST_{\d''}$. 
We can write $Q=\sum_{i=0}^{n_p} P_i x_{p,i}$ for some $P_i\in \C[\x]_{\d'}$. Since $\ev_{\Sigma,T,\d}$ is surjective, for every $i=1,\ldots,n_p$ we can find $R_i\in\C[\x]_\d$ such that $\ev_{\Sigma,T,\d}(R_i)=\ev_{\Sigma,T,\d'}(P_i)$. Then we write
\begin{align*}
Q
    &=
		    \sum_{i=0}^{n_p} 
				      P_i x_{p,i} 
				- 
				\sum_{i=1}^{n_p} 
				      R_i x_{p,0}x_{p,i}
				+
				\sum_{i=0}^{n_p} 
				      R_i x_{p,0}x_{p,i}
				\\
   &=
	      \sum_{i=1}^{n_p} 
				      x_{p,i} (P_i-x_{p,0}R_i) 
				+ 
				x_{p,0}( P_0+     \sum_{i=1}^{n_p}        R_i x_{p,i}  )
				.
\end{align*}
We notice that $P_i-x_{p,0}R_i \in \IST_{\d'}$ by construction and $Q\in \IST$. Therefore also $P_0+\sum_{i=1}^{n_p} R_i x_{p,i}$ is in $\IST_{\d'}$ and this concludes the proof.
\end{proof}

\begin{proposition}\label{IST:degree}
The subscheme $\mcl{Z}(\IST)$ is zero-dimensional and $\deg(\IST)=\abs{\Sigma}\TW$.%
\end{proposition}
\begin{proof}
By \cref{def:IST} and \cref{lemma:interpolation:trivial} we have, for $\d$ sufficiently large, the following exact sequence of $\C$-vector spaces:
\[
0\to \IST_\d \longhookrightarrow \C[\x]_\d \xrightarrow{\ev_{\Sigma,T,\d}} \C^{\abs{\Sigma}\TW}\to 0
\]
which immediately implies that the value of the Hilbert function $\dim_\C\left(\C[\x]/\IST \right)_\d$ is constantly equal to $|\Sigma|\TW$ 
 for every $\d$ sufficiently big. The degree of the attached Hilbert polynomial is therefore zero, and its only nonzero term is $|\Sigma|\TW$.
\end{proof}

\begin{proposition}\label{IST:primary:decomposition}
For every $\gamma\in\Sigma$ the ideal $I^{\{\gamma\},1}$ is prime and $I^{\{\gamma\},T}$ is $I^{\{\gamma\},1}$-primary.
The minimal primary decomposition of $\IST$ is 
$\IST=\bigcap_{\gamma\in\Sigma} I^{\{\gamma\},T}$.
\end{proposition}

\begin{proof}
$I^{\{\gamma\},1}$ is generated by the multihomogeneous polynomials vanishing at the point $\gamma$ or, in other words, is the ideal of definition for the reduced irreducible multiprojective scheme corresponding to that point. Then $I^{\{\gamma\},1}$ is a prime ideal.
From Leibnitz rule we get $(I^{\{\gamma\},1})^T\subseteq I^{\{\gamma\},T}\subseteq I^{\{\gamma\},1}$ and so the radical of $I^{\{\gamma\},T}$ is $I^{\{\gamma\},1}$. This implies that $I^{\{\gamma\},1}$ is the only minimal prime over $I^{\{\gamma\},T}$. Moreover, $I^{\{\gamma\},T}$ is multisaturated, because if $f\in\C[\x]$ is multihomogeneous and $f\C[\x]_\d\subseteq I^{\{\gamma\},T}$ for some $\d\in\N^\q$, then in particular $fz\in I^{\{\gamma\},T}$ for $z=\prod_{p=1}^\q x_{p,0}^{d_p}$ and so $f\in I^{\{\gamma\},T}$.
By multisaturation and \cref{multisaturation}, since moreover $\mcl{Z}(I^{\{\gamma\},T})$ is zero-dimensional by \cref{IST:degree}, $I^{\{\gamma\},T}$ cannot have embedded associated primes. Therefore its minimal primary decomposition consists of only one primary ideal, necessarily equal to $I^{\{\gamma\},T}$ itself.
Finally, the equality $\IST=\bigcap_{\gamma\in\Sigma} I^{\{\gamma\},T}$ is clear, and since the ideals appearing in this formula are primary ideals corresponding to distinct prime ideals without mutual inclusions, this gives an irredundant primary decomposition for $\IST$.
\end{proof}


\subsection{The main corollary}
\label{sec:application:corollary}

For this paragraph we keep the notations of \cref{sec:application:group} and we denote by $\nG$ the dimension of $G$.
The following is the corollary we aimed for.

\begin{theorem}\label{thm:res:estimate}
Let $\dd=(\mbf{d}^{(0)},\ldots,\mbf{d}^{(\nG)})$ be a collection of multidegrees such that $\ev_{\Sigma,T,\di}$ is surjective for all $i=0,\ldots,\nG-1$. 
Then the resultant  $\Rres_{\dd}(\mfk{G})$ of index $\dd$ attached to the prime ideal $\mfk{G}$ vanishes with multiplicity at least $|\Sigma|\TW$ 
 on every $(\nG+1)$-uple of polynomials in $I^{\Sigma,T}_{\mbf{d}^{(0)}}\times\dots\times I^{\Sigma,T}_{\mbf{d}^{(\nG)}}$.
\end{theorem}
\begin{proof}
By \cref{IST:generated} we have $(\IST_\di)=\IST\cap\C[\x]_{\geq\di}$ for every $i=0,\ldots,\nG-1$. 
Therefore for the same values of $i$ we have that $\mcl{Z}(\IST_\di)=\mcl{Z}(\IST)$ and, by \cref{multisaturation}, that the ideals $\IST_\di$ and $\IST$ have the same relevant associated ideals. 
By \cref{IST:primary:decomposition} these primes correspond to reduced irreducible multiprojective subschemes supported on the points of $\Sigma$. 
Since $\Sigma\subseteq G(\C)$ and $G$ is an algebraic group we see that $\mcl{Z}(\mfk{G})$ is smooth at every such point and is therefore locally a complete intersection.
$\C[\x]$ being Cohen-Macaulay at every localization, we deduce that for every relevant $\p\in\Ass_{\C[\x]} (\C[\x]/\IST)$ the local ring $(\C[\x]/\mfk{G})_\p$ is Cohen-Macaulay as well. 
The thesis is then a corollary of \cref{thm:main:roy:CM} and \cref{IST:degree}.
\end{proof}

\begin{remark}
The hypothesis of \cref{thm:res:estimate} are satisfied if the multidegrees $\di$ are large enough, thanks to the trivial estimate given in \cref{lemma:interpolation:trivial}. 
In practice, one may want to apply the theorem in an optimal situation and therefore may seek for sharper conditions that imply the surjectivity of the maps $\ev_{\Sigma,T,\di}$. 
This is exactly the objective of an interpolation lemma, for which we refer the reader, for example, to \cite{Fischler2003},  \cite{Fischler2005} or \cite{Fischler2014}.
\end{remark}



\bibliographystyle{alpha}
\bibliography{biblio_article_resultant_v3}


\end{document}